\documentclass{amsart}
\usepackage{amsmath,amssymb,amsthm,bm,bbm}
\usepackage{graphicx,enumerate}
\usepackage{layout}
\usepackage{longtable}

\theoremstyle{plain}
\newtheorem{theorem}{Theorem}[section]
\newtheorem*{theorem-nn}{Theorem}
\newtheorem{lemma}[theorem]{Lemma}

\newtheorem*{proposition-nn}{Proposition}

\theoremstyle{definition}
\newtheorem{definition}[theorem]{Definition}
\newtheorem{algorithm}[theorem]{Algorithm}
\newtheorem{example}[theorem]{Example}
\newtheorem{remark}[theorem]{Remark}
\newtheorem*{acknowledgments}{Acknowledgments}

\theoremstyle{remark}

\newcommand{\bZ}{\mathbbm{Z}}
\newcommand{\bG}{\mathbbm{G}}
\newcommand{\bF}{\mathbbm{F}}
\newcommand{\bP}{\mathbbm{P}}
\newcommand{\cC}{\mathcal{C}}\newcommand{\cD}{\mathcal{D}}
\newcommand{\cH}{\mathcal{H}}\newcommand{\cS}{\mathcal{S}}
\newcommand{\GL}{{\rm GL}}
\newcommand{\PGL}{{\rm PGL}}\newcommand{\PSL}{{\rm PSL}}
\newcounter{sub}
{\begin{list}{(\arabic{sub})}{\usecounter{sub}%
\setlength{\leftmargin}{2em}}}{\end{list}}

\newcommand{\Aut}{\mathrm{Aut}}

\title[Rationality problem for norm one tori in small dimensions]
{Rationality problem for norm one tori in small dimensions}

\author[S. Hasegawa]{Sumito Hasegawa}
\address{Graduate School of Science and Technology, Niigata University, Niigata 950-2181, Japan}

\email{shasegawa@m.sc.niigata-u.ac.jp}

\author[A. Hoshi]{Akinari Hoshi}
\address{Department of Mathematics, Niigata University, Niigata 950-2181, Japan}
\email{hoshi@math.sc.niigata-u.ac.jp}

\author[A. Yamasaki]{Aiichi Yamasaki}
\address{Department of Mathematics, Kyoto University, Kyoto 606-8502, Japan}
\email{aiichi.yamasaki@gmail.com}

\thanks{
This work was partially supported by JSPS KAKENHI Grant Numbers 
25400027, 16K05059, 19K03418.
}

\begin{document}
\begin{abstract}
We classify stably/retract rational norm one tori 
in dimension $n-1$ 
for $n=2^e$ $(e\geq 1)$ is a power of $2$ and $n=12, 14, 15$. 
Retract non-rationality of norm one tori 
for primitive $G\leq S_{2p}$ where $p$ is a prime number 
and for the five Mathieu groups $M_n\leq S_n$ 
$(n=11,12,22,23,24)$ is also given.
\end{abstract}

\maketitle

%\tableofcontents

%%%%%%%%%%%%%%%%%%%%%%%%%%%%%%%%%%%%%%%%%%%%%%%%%%%%%%%%
\section{Introduction}\label{seInt}

Let $L$ be a finite Galois extension of a field $k$ 
and $G={\rm Gal}(L/k)$ be the Galois group of the extension $L/k$. 
Let $M=\bigoplus_{1\leq i\leq n}\bZ\cdot u_i$ be a $G$-lattice with 
a $\bZ$-basis $\{u_1,\ldots,u_n\}$, 
i.e. finitely generated $\bZ[G]$-module 
which is $\bZ$-free as an abelian group. 
Let $G$ act on the rational function field $L(x_1,\ldots,x_n)$ 
over $L$ with $n$ variables $x_1,\ldots,x_n$ by 
\begin{align}
\sigma(x_i)=\prod_{j=1}^n x_j^{a_{i,j}},\quad 1\leq i\leq n\label{acts}
\end{align}
for any $\sigma\in G$, when $\sigma (u_i)=\sum_{j=1}^n a_{i,j} u_j$, 
$a_{i,j}\in\bZ$. 
The field $L(x_1,\ldots,x_n)$ with this action of $G$ will be denoted 
by $L(M)$.
%An algebraic torus $T$ over $k$ is a group $k$-scheme 
%such that $T\otimes_k \overline{k}\simeq \bG_{m,\overline{k}}^n$. 
There is the duality between the category of $G$-lattices 
and the category of algebraic $k$-tori which split over $L$ 
(see \cite[Section 1.2]{Ono61}, \cite[page 27, Example 6]{Vos98}). 
%A torus $T$ corresponds in this duality to the dual $X(T)^\circ$ 
%of the character group $X(T)={\rm Hom}(T,\bG_m)$. 
In fact, if $T$ is an algebraic $k$-torus, then the character 
group $X(T)={\rm Hom}(T,\bG_m)$ of $T$ may be regarded as a $G$-lattice. 
Conversely, for a given $G$-lattice $M$, there exists an algebraic $k$-torus 
$T$ which splits over $L$ such that $X(T)$ is isomorphic to $M$ 
as a $G$-lattice. 

The invariant field $L(M)^G$ of $L(M)$ under the action of $G$ 
may be identified with the function field of the algebraic $k$-torus $T$. 
Note that the field $L(M)^G$ is always $k$-unirational 
(see \cite[page 40, Example 21]{Vos98}).
Tori of dimension $n$ over $k$ correspond bijectively 
to the elements of the set $H^1(\mathcal{G},\GL_n(\bZ))$ 
where $\mathcal{G}={\rm Gal}(k_{\rm s}/k)$ since 
${\rm Aut}(\bG_m^n)=\GL_n(\bZ)$. 
The $k$-torus $T$ of dimension $n$ is determined uniquely by the integral 
representation $h : \mathcal{G}\rightarrow \GL_n(\bZ)$ up to conjugacy, 
and the group $h(\mathcal{G})$ is a finite subgroup of $\GL_n(\bZ)$ 
(see \cite[page 57, Section 4.9]{Vos98})). 

Let $K/k$ be a separable field extension of degree $n$ 
and $L/k$ be the Galois closure of $K/k$. 
Let $G={\rm Gal}(L/k)$ and $H={\rm Gal}(L/K)$. 
The Galois group $G$ may be regarded as a transitive subgroup of 
the symmetric group $S_n$ of degree $n$. 
Let $R^{(1)}_{K/k}(\bG_m)$ be the norm one torus of $K/k$,
i.e. the kernel of the norm map $R_{K/k}(\bG_m)\rightarrow \bG_m$ where 
$R_{K/k}$ is the Weil restriction (see \cite[page 37, Section 3.12]{Vos98}). 
The norm one torus $R^{(1)}_{K/k}(\bG_m)$ has the 
Chevalley module $J_{G/H}$ as its character module 
and the field $L(J_{G/H})^G$ as its function field 
where $J_{G/H}=(I_{G/H})^\circ={\rm Hom}_\bZ(I_{G/H},\bZ)$ 
is the dual lattice of $I_{G/H}={\rm Ker}\ \varepsilon$ and 
$\varepsilon : \bZ[G/H]\rightarrow \bZ$ is the augmentation map 
(see \cite[Section 4.8]{Vos98}). 
We have the exact sequence $0\rightarrow \bZ\rightarrow \bZ[G/H]
\rightarrow J_{G/H}\rightarrow 0$ and ${\rm rank}_\bZ(J_{G/H})=n-1$. 
Write $J_{G/H}=\oplus_{1\leq i\leq n-1}\bZ x_i$. 
Then the action of $G$ on $L(J_{G/H})=L(x_1,\ldots,x_{n-1})$ is 
of the form (\ref{acts}). 

%%%%%%%%%%%%%%%%%%%%%%%%%%%%%%%%%%%%%%%%%%%%%%%%%%%%%%

Let %$k$ be a field and 
$K$ %and $K^\prime$ 
be a finitely generated field extension of a field $k$. 
A field $K$ is called {\it rational over $k$} 
(or {\it $k$-rational} for short) 
if $K$ is purely transcendental over $k$, 
i.e. $K$ is isomorphic to $k(x_1,\ldots,x_n)$, 
the rational function field over $k$ with $n$ variables $x_1,\ldots,x_n$ 
for some integer $n$. 
$K$ is called {\it stably $k$-rational} 
if $K(y_1,\ldots,y_m)$ is $k$-rational for some algebraically 
independent elements $y_1,\ldots,y_m$ over $K$. 
Two fields 
$K$ and $K^\prime$ are called {\it stably $k$-isomorphic} if 
$K(y_1,\ldots,y_m)\simeq K^\prime(z_1,\ldots,z_n)$ over $k$ 
for some algebraically independent elements $y_1,\ldots,y_m$ over $K$ 
and $z_1,\ldots,z_n$ over $K^\prime$. 
When $k$ is an infinite field, 
%%%
$K$ is called {\it retract $k$-rational} 
if there is a $k$-algebra $R$ contained in $K$ such that 
(i) $K$ is the quotient field of $R$, and (ii) 
the identity map $1_R : R\rightarrow R$ factors through a localized 
polynomial ring over $k$, i.e. there is an element $f\in k[x_1,\ldots,x_n]$, 
which is the polynomial ring over $k$, and there are $k$-algebra 
homomorphisms $\varphi : R\rightarrow k[x_1,\ldots,x_n][1/f]$ 
and $\psi : k[x_1,\ldots,x_n][1/f]\rightarrow R$ satisfying 
$\psi\circ\varphi=1_R$ (cf. \cite{Sal84}). 
$K$ is called {\it $k$-unirational} 
if $k\subset K\subset k(x_1,\ldots,x_n)$ for some integer $n$. 
It is not difficult to see that 
``$k$-rational'' $\Rightarrow$ ``stably $k$-rational'' $\Rightarrow$ 
``retract $k$-rational'' $\Rightarrow$ ``$k$-unirational''. 

%%%%%%%%%%%%%%%%%%%%%
The $1$-dimensional algebraic $k$-tori, 
i.e. the trivial torus $\bG_m$ and the norm one torus $R_{K/k}^{(1)}(\bG_m)$ 
with $[K:k]=2$, are $k$-rational. 
%There are $13$ (resp. $73$, $710$, $6079$) $\bZ$-classes forming 
%$10$ (resp. $32$, $227$, $955$) $\bQ$-classes 
%in $\GL_2(\bZ)$ (resp. $\GL_3(\bZ)$, $\GL_4(\bZ)$, $\GL_5(\bZ)$).
Voskresenskii \cite{Vos67} showed that 
all the $2$-dimensional algebraic $k$-tori are $k$-rational. 
Kunyavskii \cite{Kun90} classifies 
rational (resp. stably rational, retract rational) algebraic $k$-tori 
in dimension $3$. 
Hoshi and Yamasaki \cite[Theorem 1.9, Theorem 1.12]{HY17} 
classify stably rational (resp. retract rational) 
algebraic $k$-tori in dimensions $4$ and $5$. 

%%%%
%Let $T=R^{(1)}_{K/k}(\bG_m)$ be the norm one torus defined by $K/k$. 
Let $S_n$ (resp. $A_n$, $D_n$, $C_n$) be the symmetric 
(resp. the alternating, the dihedral, the cyclic) group 
of degree $n$ of order $n!$ (resp. $n!/2$, $2n$, $n$). % $n$).
Let $F_{pm}\simeq C_p\rtimes C_m\leq S_p$ 
be the Frobenius group of order $pm$ where $m\mid p-1$.
%%%%%%%%%%%%%%%%%%%
Let $nTm$ be the $m$-th transitive subgroup of $S_n$ 
%There exist $2$ (resp. $5$, $5$, $16$, $7$, $50$, $34$, $45$, $8$) 
%transitive subgroups of $S_3$ (resp. $S_4$, $S_5$, $S_6$, $S_7$, 
%$S_8$, $S_9$, $S_{10}$, $S_{11}$) 
(see Butler and McKay \cite{BM83} for $n\leq 11$, 
Royle \cite{Roy87} for $n=12$, 
Butler \cite{But93} for $n=14,15$ 
and \cite{GAP}).
%

%%%%%%%%%%%%%%%%%%%%%%%%%%%%%%%%%%%%%%%%%%%%%%%%%%
The rationality problem for norm one tori 
$R^{(1)}_{K/k}(\bG_m)$ 
is investigated 
by \cite{EM75}, \cite{CTS77}, \cite{Hur84}, \cite{CTS87}, 
\cite{LeB95}, \cite{CK00}, \cite{LL00}, \cite{Flo}, \cite{End11}, 
\cite{HY17} and \cite{HY}. 
In the previous papers \cite{HY17} and \cite{HY}, 
a classification of stably/retract rational
norm one tori $R^{(1)}_{K/k}(\bG_m)$ in dimension $p-1$ 
where $p$ is a prime number 
and in dimension $n\leq 10$ is given 
except for the following three cases: 
(i) $G=\PSL_2(\bF_{2^e})$ where 
$p=2^e+1\geq 17$ is a Fermat prime; 
(ii) $G=9T27\simeq \PSL_2(\bF_8)$; 
(iii) $G=10T11\simeq A_5\times C_2$. 

%The following three theorems are main results of this paper. 
The first main results of this paper are 
Theorem \ref{thmain1} and Theorem \ref{thmain2} 
which classify stably/retract rational norm one tori 
$R^{(1)}_{K/k}(\bG_m)$ in dimension $n-1$ 
for $n=2^e$ $(e\geq 1)$ and $n=10,12,14,15$. 
Note that there exist $45$ (resp. $301$, $63$, $104$) transitive groups $10Tm$ 
(resp. $12Tm$, $14Tm$, $15Tm$) of degree $10$ (resp. $12$, $14$, $15$). 
The case $n=10$ in Theorem \ref{thmain2} (1) 
was solved by \cite[Theorem 1.11]{HY} except for 
$G=10T11\simeq A_5\times C_2$. 

%%%%%%%%%%%%%%%%%%%%%%%%%%%%
\begin{theorem}\label{thmain1}
Let $K/k$ be a separable field extension of degree $n$ 
and $L/k$ be the Galois closure of $K/k$. 
%Assume that 
Let $G={\rm Gal}(L/k)$ be a transitive subgroup of $S_n$ 
where $n=2^e$ $(e\geq 1)$ 
and $H={\rm Gal}(L/K)$ with $[G:H]=n$. 
%is the stabilizer of one of the letters in $G$. 
Then $R_{K/k}^{(1)}(\bG_m)$ is stably $k$-rational 
if and only if $G\simeq C_n$. 
Moreover, if $R_{K/k}^{(1)}(\bG_m)$ is not stably $k$-rational, 
then it is not retract $k$-rational. 
\end{theorem}
\begin{theorem}\label{thmain2}
Let $K/k$ be a separable field extension of degree $n$ 
and $L/k$ be the Galois closure of $K/k$. 
%Assume that 
Let $G={\rm Gal}(L/k)$ be a transitive subgroup of $S_n$ 
%which acts on $L(x_1,x_2,x_3,x_4)$ via $(\ref{acts})$, 
and $H={\rm Gal}(L/K)$ with $[G:H]=n$.
%is the stabilizer of one of the letters in $G$. 
Then %a birational 
a classification of %the 
stably/retract rational norm one tori 
$T=R_{K/k}^{(1)}(\bG_m)$ in dimension $n-1$ 
for $n=10,12,14,15$ is given as follows:\\
{\rm (1)} The case $10Tm$ $(1\leq m\leq 45)$.\\
{\rm (i)} $T$ is stably $k$-rational for 
$10T1\simeq C_{10}$, $10T2\simeq D_5$, $10T3\simeq D_{10}$, 
$10T11\simeq A_5\times C_2$;\\ 
{\rm (ii)} $T$ is not stably but retract $k$-rational for 
$10T4\simeq F_{20}$, $10T5\simeq F_{20}\times C_2$, 
$10T12\simeq S_5$, $10T22\simeq S_5\times C_2$;\\
{\rm (iii)} $T$ is not retract $k$-rational for $10Tm$ with 
$6\leq m\leq 45\ \textrm{and}\ m\neq 11,12,22$.\\
{\rm (2)} The case $12Tm$ $(1\leq m\leq 301)$.\\
{\rm (i)} $T$ is stably $k$-rational for $12T1 \simeq C_{12}$, $12T5 \simeq C_3 \rtimes C_4$, $12T11 \simeq C_4 \times S_3$;\\
{\rm (ii)} $T$ is not retract $k$-rational for $12Tm$ with $1\leq m\leq 301$ and $m \not = 1,5,11$.\\
{\rm (3)} The case $14Tm$ $(1\leq m\leq 63)$.\\
{\rm (i)} $T$ is stably $k$-rational for $14T1 \simeq C_{14}$, $14T2 \simeq D_{7}$, $14T3 \simeq D_{14}$;\\
{\rm (ii)} $T$ is not stably $k$-rational but retract $k$-rational for $14T4 \simeq F_{42}$, $14T5 \simeq F_{21} \times C_2$, $14T7 \simeq F_{42} \times C_2$, $14T16 \simeq \PSL_3(\bF_2) \rtimes C_2$, $14T19 \simeq  \PSL_3(\bF_2) \times C_2$, $14T46 \simeq S_{7}$, $14T47 \simeq A_{7} \times C_2$, $14T49 \simeq S_{7} \times C_2$;\\
{\rm (iii)} $T$ is not retract $k$-rational for $14Tm$ with $6 \leq m \leq 63$ and $m \not = 7, 16,19,46,47,49$.\\
{\rm (4)} The case $15Tm$ $(1\leq m \leq 104)$.\\
{\rm (i)} $T$ is stably $k$-rational for $15T1 \simeq C_{15}$, $15T2 \simeq D_{15}$, $15T3 \simeq D_5 \times C_3$, $15T4 \simeq S_3 \times C_5$, $15T5 \simeq A_5$, $15T7 \simeq D_5 \times S_3$ $15T16 \simeq A_5 \times C_3 \simeq \GL_2(\bF_4)$, $15T23 \simeq A_5 \times S_3$;\\
{\rm (ii)} $T$ is not stably $k$-rational but retract $k$-rational for $15T6 \simeq C_{15} \rtimes C_4$, $15T8 \simeq F_{20} \times C_3$, $15T10 \simeq S_5$, $15T11 \simeq F_{20} \times S_3$, $15T22 \simeq (A_5 \times C_3) \rtimes C_2 \simeq \GL_2(\bF_4) \rtimes C_2$, $15T24 \simeq S_5 \times C_3$, $15T29 \simeq S_5 \times S_3$;\\
{\rm (iii)} $T$ is not retract $k$-rational for $15Tm$ with $9 \leq m \le104$ and $m \not = 10,11,16,22,23,24,29$.
%{\rm (4)} The case $16Tm$ $(1\leq m\leq 1954)$.\\
%{\rm (i)} $T$ is stably $k$-rational for $16T1 \simeq C_{16}$;\\
%{\rm (ii)} $T$ is not retract $k$-rational for $16Tm$ with 
%$2 \leq m \leq 1954$.
\end{theorem}

The second main result of this paper is the following:
\begin{theorem}\label{thmain3}
Let $K/k$ be a separable field extension of degree $n$ 
and $L/k$ be the Galois closure of $K/k$. 
Let $G={\rm Gal}(L/k)$ be a transitive subgroup of $S_n$ 
and $H={\rm Gal}(L/K)$ with $[G:H]=n$.
%be the stabilizer of one of the letters in $G$. 
Assume that $n=q+1$ where $q=l^e\equiv 1\pmod{4}$ 
is an odd prime power 
and 
$\PSL_2(\bF_q)\leq G\leq {\rm P\Gamma L}_2(\bF_q)\simeq 
\PGL_2(\bF_q)\rtimes C_e$. 
%{\rm Gal}(\bF_q/\bF_l)$. 
Then $R_{K/k}^{(1)}(\bG_m)$ is not retract $k$-rational. 
\end{theorem}
As a consequence of Theorem \ref{thmain3}, we will show 
Theorem \ref{thmain4} which gives a classification of 
stably/retract rational norm one tori $R^{(1)}_{K/k}(\bG_m)$ in 
dimension $n-1$ where $n=2p$, $p$ is a prime number and 
$G={\rm Gal}(L/k)\leq S_{2p}$ is primitive. 
\begin{theorem}\label{thmain4}
Let $p$ be a prime number, 
$K/k$ be a separable field extension of degree $2p$ 
and $L/k$ be the Galois closure of $K/k$. 
Assume that $G={\rm Gal}(L/k)$ is a primitive %transitive 
subgroup of $S_{2p}$ 
and $H={\rm Gal}(L/K)$ with $[G:H]=2p$.
%is the stabilizer of one of the letters in $G$. 
Then $R_{K/k}^{(1)}(\bG_m)$ is not retract $k$-rational. 

More precisely, $R_{K/k}^{(1)}(\bG_m)$ is 
not retract $k$-rational for the following primitive %transitive 
groups $G\leq S_{2p}$:\\ 
{\rm (i)} $G=S_{2p}$ or $G=A_{2p}\leq S_{2p}$;\\
{\rm (ii)} $G=S_5\leq S_{10}$ or $G=A_5\leq S_{10}$;\\
{\rm (iii)} $G=M_{22}\leq S_{22}$ 
or $G=\Aut(M_{22})\simeq M_{22}\rtimes C_2\leq S_{22}$ 
where $M_{22}$ is the Mathieu group of degree $22$;\\ 
{\rm (iv)} $\PSL_2(\bF_q)\leq G\leq {\rm P\Gamma L}_2(\bF_q)\simeq 
\PGL_2(\bF_q)\rtimes C_e$ %{\rm Gal}(\bF_q/\bF_l)$ 
%where $2p=q+1$ and $q=l^{2^a}$ for some odd prime number $l$ and $a\geq 1$.
where $2p=q+1$ and $q=l^e$ is an odd prime power.
\end{theorem}
%%%%%%%%%%%%%%%%%%%%%%%%%%%%%%%%%%%%%%%%%

%
\begin{remark}
For the reader's convenience, 
we give a list of non-solvable primitive groups 
$G=nTm\leq S_n$ of degree $n=10, 12, 14, 15$:\\ 
(i) %$n=10$;
$10T7\simeq A_5$,
$10T13\simeq S_5$,
$10T26\simeq \PSL_2(\bF_9)\simeq A_6$,
$10T30\simeq \PGL_2(\bF_9)$,
$10T31\simeq M_{10}$,
$10T32\simeq S_6$,
$10T35\simeq {\rm P}\Gamma L_2(\bF_9)$,
$10T44\simeq A_{10}$,
$10T45\simeq S_{10}$.\\
(ii) %$n=12$; 
$12T179\simeq \PSL_2(\bF_{11})$, 
$12T218\simeq \PGL_2(\bF_{11})$, 
$12T272\simeq M_{11}$, 
$12T295\simeq M_{12}$, 
$12T300\simeq A_{12}$, 
$12T301\simeq S_{12}$.\\
(iii) %$n=14$; 
$14T30\simeq \PSL_2(\bF_{13})$, 
$14T39\simeq \PGL_2(\bF_{13})$, 
$14T62\simeq A_{14}$, 
$14T63\simeq S_{14}$.\\
(iv) %$n=15$;
$15T20\simeq A_6$, 
$15T28\simeq S_6$, 
$15T47\simeq A_7$, 
$15T72\simeq A_8\simeq \PSL_4(\bF_2)$, 
$15T103\simeq A_{15}$, 
$15T104\simeq S_{15}$.
%10
%[7,13,26,30,31,32,35,44,45]
%[ A_5(10), S_5(10d), L(10)=PSL(2,9), L(10):2=PGL(2,9), 
%M(10)=L(10)'2, S_6(10)=L(10):2, L(10).2^2=P|L(2,9), A10, S10 ]
%12
%[179,218,272,295,300,301]
%[ L(2,11), PGL(2,11), M_11(12), M(12), A12, S12 ]
%14
%[30,39,62,63]
%[ L(14)=PSL(2,13), L(14):2=PGL(2,13), A14, S14 ]
%15
%[20,28,47,72,103,104]
%[ A_6(15), S_6(15), A_7(15), L(15)=A_8(15)=PSL(4,2), A15, S15 ]
\end{remark}

We also give the following result for the five 
Mathieu groups $M_n\leq S_n$ where $n=11,12,22,23,24$:
\begin{theorem}\label{thmain5}
Let $K/k$ be a separable field extension of degree $n$ 
and $L/k$ be the Galois closure of $K/k$. 
Let $G={\rm Gal}(L/k)$ be a transitive subgroup of $S_n$ 
and $H={\rm Gal}(L/K)$ with $[G:H]=n$.
%be the stabilizer of one of the letters in $G$. 
Assume that $n=11$, $12$, $22$, $23$ or $24$ 
and $G$ is isomorphic to the Mathieu group $M_n$ 
%$M_{11}$ $($resp. $M_{12}, M_{22}, M_{23}, M_{24}$$)$ 
of degree $n$. 
Then $R_{K/k}^{(1)}(\bG_m)$ is not retract $k$-rational.
\end{theorem}
\medskip

We organize this paper as follows. 
In Section \ref{sePre}, we prepare some basic tools to prove 
stably/retract rationality of algebraic tori.  
In Section \ref{seProof1}, 
we will give the proof of Theorem \ref{thmain1}. 
In Section \ref{seProof2}, we will give the proof of 
Theorem \ref{thmain2}. 
Finally, we give the proof of 
Theorem  \ref{thmain3}, Theorem \ref{thmain4} 
and Theorem \ref{thmain5} in Section \ref{seProof34}. 

We note that the proofs of Theorem \ref{thmain2}, Theorem \ref{thmain4} 
and Theorem \ref{thmain5} 
are given by applying GAP algorithms 
which are available from 
{\tt https://www.math.kyoto-u.ac.jp/\~{}yamasaki/Algorithm/RatProbNorm1Tori/}
although the proofs of Theorem \ref{thmain1} and 
Theorem \ref{thmain3} are given by purely algebraic way.
\begin{acknowledgments}
The authors would like to thank 
Ming-chang Kang and Shizuo Endo 
for giving them useful and valuable comments. 
They also thank the referee and the editor 
for crucial advice to organize 
the paper with the aid of computer algorithms.
\end{acknowledgments}

%
%%%%%%%%%%%%%%%%%%%%%%%%%%%%%%%%%%%%%%%%%%%%%%%%%%%%%%%%%%%%%%%%%%%%%%%%
\section{Preliminaries: rationality problem for algebraic tori and flabby resolution}\label{sePre}

We recall some basic facts of the theory of flabby (flasque) $G$-lattices
(see Colliot-Th\'{e}l\`{e}ne and Sansuc \cite{CTS77}, 
Swan \cite{Swa83}, 
Voskresenskii \cite[Chapter 2]{Vos98}, 
Lorenz \cite[Chapter 2]{Lor05}, 
Swan \cite{Swa10}).

\begin{definition}
%[Permutation, stably permutation, invertible, flabby and coflabby $G$-lattices]
Let $G$ be a finite group and $M$ be a $G$-lattice 
(i.e. finitely generated $\bZ[G]$-module which is $\bZ$-free 
as an abelian group). \\
{\rm (i)} $M$ is called a {\it permutation} $G$-lattice 
if $M$ has a $\bZ$-basis permuted by $G$, 
i.e. $M\simeq \oplus_{1\leq i\leq m}\bZ[G/H_i]$ 
for some subgroups $H_1,\ldots,H_m$ of $G$.\\
{\rm (ii)} $M$ is called a {\it stably permutation} $G$-lattice 
if $M\oplus P\simeq P^\prime$ 
for some permutation $G$-lattices $P$ and $P^\prime$.\\
{\rm (iii)} $M$ is called {\it invertible} (or {\it permutation projective}) 
if it is a direct summand of a permutation $G$-lattice, 
i.e. $P\simeq M\oplus M^\prime$ for some permutation $G$-lattice 
$P$ and a $G$-lattice $M^\prime$.\\
{\rm (iv)} $M$ is called {\it flabby} (or {\it flasque}) if $\widehat H^{-1}(H,M)=0$ 
for any subgroup $H$ of $G$ where $\widehat H$ is the Tate cohomology.\\
{\rm (v)} $M$ is called {\it coflabby} (or {\it coflasque}) if $H^1(H,M)=0$
for any subgroup $H$ of $G$.
\end{definition}

\begin{lemma}[Lenstra {\cite[Propositions 1.1 and 1.2]{Len74}, see also Swan 
\cite[Section 8]{Swa83}}]\label{lemSL}
Let $E$ be an invertible $G$-lattice.\\
{\rm (i)} $E$ is flabby and coflabby.\\
{\rm (ii)} If $C$ is a coflabby $G$-lattice, then any short exact sequence
$0 \rightarrow C \rightarrow N \rightarrow E \rightarrow 0$ splits.
\end{lemma}
%%%%%%%%%%%%%%%%%%%%%%%%%%

\begin{definition}[{see \cite[Section 1]{EM75}, \cite[Section 4.7]{Vos98}}]
Let $\cC(G)$ be the category of all $G$-lattices. 
Let $\cS(G)$ be the full subcategory of $\cC(G)$ of all permutation $G$-lattices 
and $\cD(G)$ be the full subcategory of $\cC(G)$ of all invertible $G$-lattices.
Let 
\begin{align*}
\cH^i(G)=\{M\in \cC(G)\mid \widehat H^i(H,M)=0\ {\rm for\ any}\ H\leq G\}\ (i=\pm 1)
\end{align*}
be the class of ``$\widehat H^i$-vanish'' $G$-lattices 
where $\widehat H^i$ is the Tate cohomology. 
Then we have the inclusions 
$\cS(G)\subset \cD(G)\subset \cH^i(G)\subset \cC(G)$ $(i=\pm 1)$. 
\end{definition}
\begin{definition}\label{defCM}
We say that two $G$-lattices $M_1$ and $M_2$ are {\it similar} 
if there exist permutation $G$-lattices $P_1$ and $P_2$ such that 
$M_1\oplus P_1\simeq M_2\oplus P_2$. 
We denote the similarity class of $M$ by $[M]$. 
The set of similarity classes $\cC(G)/\cS(G)$ becomes a 
commutative monoid 
(with respect to the sum $[M_1]+[M_2]:=[M_1\oplus M_2]$ 
and the zero $0=[P]$ where $P\in \cS(G)$). 
\end{definition}
\begin{theorem}[{Endo and Miyata \cite[Lemma 1.1]{EM75}, Colliot-Th\'el\`ene and Sansuc \cite[Lemma 3]{CTS77}, 
see also \cite[Lemma 8.5]{Swa83}, \cite[Lemma 2.6.1]{Lor05}}]\label{thEM}
For any $G$-lattice $M$,
there exists a short exact sequence of $G$-lattices
$0 \rightarrow M \rightarrow P \rightarrow F \rightarrow 0$
where $P$ is permutation and $F$ is flabby.
\end{theorem}
\begin{definition}\label{defFlabby}
The exact sequence $0 \rightarrow M \rightarrow P \rightarrow F \rightarrow 0$ 
as in Theorem \ref{thEM} is called a {\it flabby resolution} of the $G$-lattice $M$.
$\rho_G(M)=[F] \in \cC(G)/\cS(G)$ is called {\it the flabby class} of $M$,
denoted by $[M]^{fl}=[F]$.
Note that $[M]^{fl}$ is well-defined: 
if $[M]=[M^\prime]$, $[M]^{fl}=[F]$ and $[M^\prime]^{fl}=[F^\prime]$
then $F \oplus P_1 \simeq F^\prime \oplus P_2$
for some permutation $G$-lattices $P_1$ and $P_2$,
and therefore $[F]=[F^\prime]$ (cf. \cite[Lemma 8.7]{Swa83}). 
We say that $[M]^{fl}$ is {\it invertible} if 
$[M]^{fl}=[E]$ for some invertible $G$-lattice $E$. 
\end{definition}

For $G$-lattice $M$, 
it is not difficult to see 
\begin{align*}
\textrm{permutation}\ \ 
\Rightarrow\ \ 
&\textrm{stably\ permutation}\ \ 
\Rightarrow\ \ 
\textrm{invertible}\ \ 
\Rightarrow\ \ 
\textrm{flabby\ and\ coflabby}\\
&\hspace*{8mm}\Downarrow\hspace*{34mm} \Downarrow\\
&\hspace*{7mm}[M]^{fl}=0\hspace*{10mm}\Rightarrow\hspace*{5mm}[M]^{fl}\ 
\textrm{is\ invertible}.
\end{align*}

The above implications in each step cannot be reversed 
(see, for example, \cite[Section 1]{HY17}). 

Let $L/k$ be a finite Galois extension with Galois group $G={\rm Gal}(L/k)$ 
and $M$ be a $G$-lattice. 
The flabby class $\rho_G(M)=[M]^{fl}$ 
plays crucial role in the rationality problem for 
$L(M)^G$ as follows (see Voskresenskii's fundamental book \cite[Section 4.6]{Vos98} and Kunyavskii \cite{Kun07}, see also e.g. Swan \cite{Swa83}, 
Kunyavskii \cite[Section 2]{Kun90}, 
Lemire, Popov and Reichstein \cite[Section 2]{LPR06}, 
Kang \cite{Kan12}, 
Yamasaki \cite{Yam12}):  
\begin{theorem}[Endo and Miyata, Voskresenskii, Saltman]\label{thEM73}
Let $L/k$ be a finite Galois extension with Galois group $G={\rm Gal}(L/k)$. 
Let $M$ and $M^\prime$ be $G$-lattices.\\
{\rm (i)} $(${\rm Endo and Miyata} \cite[Theorem 1.6]{EM73}$)$ 
$[M]^{fl}=0$ if and only if $L(M)^G$ is stably $k$-rational.\\
{\rm (ii)} $(${\rm Voskresenskii} \cite[Theorem 2]{Vos74}$)$ 
$[M]^{fl}=[M^\prime]^{fl}$ if and only if $L(M)^G$ and $L(M^\prime)^G$ 
are stably $k$-isomorphic.\\
%i.e. there exist algebraically independent 
%elements $x_1,\ldots,x_m$ over $L(M)^G$ and 
%$y_1,\ldots,y_n$ over $L(M^\prime)^G$ such that 
%$L(M)^G(x_1,\ldots,x_m)\simeq L(M^\prime)^G(y_1,\ldots,y_n)$.\\
{\rm (iii)} $(${\rm Saltman} \cite[Theorem 3.14]{Sal84}$)$ 
$[M]^{fl}$ is invertible if and only if $L(M)^G$ is 
retract $k$-rational.
\end{theorem}
\begin{lemma}[{Swan \cite[Lemma 3.1]{Swa10}}]\label{lemSwa}
Let $0\rightarrow M_1\rightarrow M_2\rightarrow M_3\rightarrow 0$ 
be a short exact sequence of $G$-lattices with $M_3$ invertible. 
Then the flabby class $[M_2]^{fl}=[M_1]^{fl}+[M_3]^{fl}$. 
In particular, if $[M_1]^{fl}$ is invertible, 
then $-[M_1]^{fl}=[[M_1]^{fl}]^{fl}$. 
\end{lemma}
\begin{definition}\label{defMG} 
Let $G$ be a finite subgroup of $\GL_n(\bZ)$. 
{\it The $G$-lattice $M_G$ with ${\rm rank}_\bZ(M_G)=n$} 
is defined to be the $G$-lattice with a $\bZ$-basis $\{u_1,\ldots,u_n\}$ 
on which $G$ acts by $\sigma(u_i)=\sum_{j=1}^n a_{i,j}u_j$ for any $
\sigma=[a_{i,j}]\in G$. 
\end{definition}
\begin{lemma}[{see \cite[Remarque R2, page 180]{CTS77}, \cite[Lemma 2.17]{HY17}}]\label{lemp3}
Let $G$ be a finite subgroup of $\GL_n(\bZ)$ 
and $M_G$ be the corresponding $G$-lattice 
as in Definition \ref{defMG}. 
Let $H\leq G$ and $\rho_H(M_H)$ be the flabby class of $M_H$ 
as an $H$-lattice.\\
{\rm (i)} If $\rho_G(M_G)=0$, then $\rho_H(M_H)=0$.\\
{\rm (ii)} If $\rho_G(M_G)$ is invertible, then $\rho_H(M_H)$ is invertible.
\end{lemma}
%

%%%%%%%%%%%%%%%%%%%%%%%%%%%%%%%%%%%%%%%%%%%%
\section{Proof of Theorem {\ref{thmain1}}}\label{seProof1}

In order to prove Theorem \ref{thmain1}, we show the following two theorems. 

\begin{theorem}\label{th1}
Let $n=p^e$ be a prime power and 
$G$ be a transitive subgroup of $S_n$. 
Let $G_p={\rm Syl}_p(G)$ be a $p$-Sylow subgroup of $G$. 
Then $G_p$ is a transitive subgroup of $S_n$. 
\end{theorem}
\begin{proof}
Let $H$ be the stabilizer of one of the letters in $G$ 
and $H_p$ be a $p$-Sylow subgroup of $H$ 
with $H_p\leq G_p$. 
Because $[G:H]=n$ and $p$ does not divide 
both $[H:H_p]$ and $[G:G_p]$, 
we have $[G_p:H_p]=n=p^e$. 
Hence $H_p=G_p\cap H$ becomes the stabilizer of one of 
the letters in $G_p$ and $G_p\leq S_n$ is transitive.
\end{proof}
\begin{theorem}\label{th2}
Let $n=2^e$ be a power of $2$ and 
$G$ be a transitive subgroup of $S_n$. 
Let $G_2={\rm Syl}_2(G)$ be a $2$-Sylow subgroup of $G$. 
If $G_2\simeq C_n$, then $G\simeq C_n$. 
\end{theorem}
\begin{proof}
Let $H$ be the stabilizer of one of the letters in $G$. 
We should show that $H=1$ because $[G:H]=n$. 
We will prove $H=1$ by induction in $e$. 
When $e=1$, the assertion holds. 
For $e$, we assume that $G_2=\langle\sigma\rangle\simeq C_n$ where $n=2^e$. 
Without loss of generality, we may assume that 
$\sigma=(1\cdots n)\in S_n$. 

There exist $(n-1)!$ elements of order $n$ in $S_n$ 
which are conjugate in $S_n$. 
Let $Z_{S_n}(G_2)$ be 
the centralizer of $G_2$ in $S_n$ 
and $N_{S_n}(G_2)$ be 
the normalizer of $G_2$ in $S_n$. 
Then we see that $Z_{S_n}(G_2)=G_2\simeq C_n$ and 
$N_{S_n}(G_2)=C_n\rtimes {\rm Aut}(C_n)\simeq 
\bZ/2^e\bZ\rtimes(\bZ/2^e\bZ)^\times$. 
We also have $G_2=Z_G(G_2)\leq N_G(G_2)\leq G$. 
Because $N_G(G_2)$ is also a $2$-group, 
we obtain that $Z_G(G_2)=N_G(G_2)=G_2$. 

Let $A=\{x\in G\mid {\rm ord}(x)=n\}$ 
be the set of elements of order $n$ in $G$
and 
$A_2=\{x\in G_2\mid {\rm ord}(x)=n\}=\{\sigma^i\mid i {\rm : odd}\}$ 
be the set of elements of order $n$ in $G_2$. 
If $g\in G_2$, then $gag^{-1}=a$ for any $a\in A_2$. 
If $g\in G\setminus G_2$, then $gA_2g^{-1}\cap A_2=\emptyset$ 
because $N_G(G_2)=G_2$. 
Note that $g_1A_2g_1^{-1}=g_2A_2g_2^{-1}$ if and only if 
$g_2^{-1}g_2\in G_2$. 
Hence we have $|A|=|A_2|\cdot [G:G_2]=2^{e-1}\cdot |H|=|G|/2$. 
This implies that $A=\{x\in G\mid {\rm sgn}(x)=-1\}$. 

We claim that if $h(j)=k$ $(h\in H)$, then $j\equiv k\pmod{2}$. 
Suppose not. Then there exists $\sigma^{j-k}\in A_2$ such that 
$\sigma^{j-k}h(j)=j$. But this is impossible 
because ${\rm sgn}(\sigma^{j-k}h)=-1$ and hence 
${\rm ord}(\sigma^{j-k}h)=n$. 
%for $\sigma^{j-k}\in A_2$ and $h\in H$. 
This claim implies that $\langle \sigma^2,H\rangle$ acts on $2\bZ/n\bZ=\{2,4,\ldots,n\}$. 

On the other hand, $\langle\sigma^2,H\rangle\leq G\cap A_n$ 
because ${\rm sgn}(\sigma^2)={\rm sgn}(h)=1$ $(h\in H)$. 
We also see $\langle\sigma^2,H\rangle=G\cap A_n$ 
because $[\langle\sigma^2,H\rangle:H]=n/2$. 

Remember that $|H|=[G:G_2]$ is odd. 
The restriction $G\cap A_n|_{2\bZ/n\bZ}$ of $G\cap A_n$ 
into $2\bZ/n\bZ$ seems to be a transitive subgroup of 
$S_{2\bZ/n\bZ}=S_{\{2,4,\ldots,n\}}$ 
whose $2$-Sylow subgroup is $\langle\sigma^2\rangle|_{2\bZ/n\bZ}$. 
By the assumption of induction, we have $H|_{2\bZ/n\bZ}=1$. 
Similarly, we get $H|_{1+2\bZ/n\bZ}=1$. 
Therefore, we conclude that $H=1$. 
\end{proof}

{\it Proof of Theorem \ref{thmain1}.}  
Take a transitive subgroup 
$G={\rm Gal}(L/k)\leq S_n$ $(n=2^e)$ 
and $H={\rm Gal}(L/K)$ with $[G:H]=n$. 
%the stabilizer $H={\rm Gal}(L/K)$ of one of the letters in $G$. 
By Theorem \ref{th1}, 
the $2$-Sylow subgroup $G_2={\rm Syl}_2(G)$ of $G$ 
is a transitive subgroup of $S_n$. 

$(\Rightarrow)$ 
Assume that $G\not\simeq C_n$. 
By Theorem \ref{th2}, we have $G_2\not\simeq C_n$. 
Hence $[J_{G_2/H_2}]^{fl}$ is not invertible 
by 
%Theorem \ref{th13-1} 
Endo and Miyata 
\cite[Theorem 1.5]{EM75} 
and 
%Theorem \ref{th14} 
Endo 
\cite[Theorem 2.1]{End11} 
where $H_2$ is the $2$-Sylow subgroup of $H$. 
Because $G_2$ is transitive in $S_n$, 
it follows from Lemma \ref{lemp3} (ii) that 
$[J_{G/H}]^{fl}$ is not invertible. 
Hence 
$R_{K/k}^{(1)}(\bG_m)$ is not retract $k$-rational. 

$(\Leftarrow)$ 
By 
%Theorem \ref{th13-2}, 
Endo and Miyata 
\cite[Theorem 2.3]{EM75}, 
if $G\simeq C_n$, then 
$R_{K/k}^{(1)}(\bG_m)$ is stably $k$-rational. 
\qed

\begin{example}[The case $nTm\leq S_n$ where $n=2^e$]

(1) When $n=4$, there exist $5$ transitive subgroups 
$4Tm\leq S_4$ $(1\leq m\leq 5)$: 
$4T1\simeq C_4$, 
$4T2\simeq C_2\times C_2$, 
$4T3\simeq D_4$, 
$4T4\simeq A_4$, 
$4T5\simeq S_4$. 

(2) When $n=8$, there exist $50$ transitive subgroups of 
$8Tm\leq S_8$ $(1\leq m\leq 50)$. 
There exist $5$ groups $G=8Tm$ $(1\leq m\leq 5)$ 
with $|G|=8$ (see Butler and McKay \cite{BM83}, \cite{GAP}): 
$8T1\simeq C_8$, 
$8T2\simeq C_4\times C_2$, 
$8T3\simeq (C_2)^3$, 
$8T4\simeq D_4$, 
$8T5\simeq Q_8$.

(3) When $n=16$, there exist $1954$ transitive subgroups of 
$16Tm\leq S_{16}$ $(1\leq m\leq 1954)$. 
There exist $14$ groups $G=16Tm$ $(1\leq m\leq 14)$ 
with $|G|=16$ (see Example \ref{ex16}): 
$16T1 \simeq C_{16}$, 
$16T2 \simeq C_4\times (C_2)^2$, 
$16T3 \simeq (C_2)^4$, 
$16T5 \simeq C_4\times C_4$, 
$16T5 \simeq C_8\times C_2$, 
$16T6 \simeq M_{16}$, 
$16T7 \simeq Q_8\times C_2$, 
$16T8 \simeq C_4\rtimes C_4$, 
$16T9 \simeq D_4\times C_2$, 
$16T10 \simeq (C_4 \times C_2) \rtimes C_2$, 
$16T11 \simeq (C_4 \times C_2) \rtimes C_2$, 
$16T12 \simeq QD_8$, 
$16T13 \simeq D_8$, 
$16T14 \simeq Q_{16}$.

(4) When $n=32$, there exist $2801324$ transitive subgroups of 
$32Tm\leq S_{32}$ $(1\leq m\leq 2801324)$ (see Cannon and Holt \cite{CH08}). 
\end{example}

%%%%%%%%%
%%%%%%%%%%Computation16Tm
\begin{example}[Computations for $16Tm\leq S_{16}$]\label{ex16}
For $G=16Tm\leq S_{16}$, Theorem \ref{th1} and Theorem \ref{th2}  
can be checked by GAP as follows: 
\begin{verbatim}
gap> NrTransitiveGroups(16); # the number of transitive subgroups G=16Tm <= S16
1954
gap> Sy162:=List([1..1954],x->SylowSubgroup(TransitiveGroup(16,x),2));;
gap> Filtered([1..1954],x->IsTransitive(Sy162[x])=false); 
# all 2-Syllow subgroups of 16Tm are transitive
[  ]
gap> Filtered([1..1954],x->IsCyclic(Sy162[x])=true); 
# all 2-Syllow subgroups of 16Tm are cyclic except for m=1
[ 1 ]
gap> Filtered([1..1954],x->Size(TransitiveGroup(16,x))=16); # 16Tm with |16Tm|=16
[ 1, 2, 3, 4, 5, 6, 7, 8, 9, 10, 11, 12, 13, 14 ]
gap> List([1..14],x->StructureDescription(TransitiveGroup(16,x)));
[ "C16", "C4 x C2 x C2", "C2 x C2 x C2 x C2", "C4 x C4", "C8 x C2", "C8 : C2", 
"C2 x Q8", "C4 : C4", "C2 x D8", "(C4 x C2) : C2", "(C4 x C2) : C2", "QD16", 
 "D16", "Q16" ]
\end{verbatim}
\end{example}

%%%%%%%%%%%%%%%%%%%%%%%%%%%%%%%%%%%%%%%%%%%%%%%%%%%%%%%%%%%%%%%%%
\section{Proof of Theorem {\ref{thmain2}}}\label{seProof2}

%%%%%%%%%%%%%%%%%%%%%%%%%%%%%%%%%%%%%%%%%%%%%%%%%%%%%%%
%{\it Proof of Theorem \ref{thmain2}.}

Let $K/k$ be a separable field extension of degree $n$ 
and $L/k$ be the Galois closure of $K/k$. 
Let $G={\rm Gal}(L/k)$ be a transitive subgroup of $S_n$ 
%which acts on $L(x_1,x_2,x_3,x_4)$ via $(\ref{acts})$, 
and $H={\rm Gal}(L/K)$ with $[G:H]=n$. 
%%%%%%
We may assume that 
$H$ is the stabilizer of one of the letters in $G$, 
i.e. $L=k(\theta_1,\ldots,\theta_n)$ and $K=L^H=k(\theta_i)$ 
where $1\leq i\leq n$. 

Let $nTm$ be the $m$-th transitive subgroup of $S_n$ 
%There exist $2$ (resp. $5$, $5$, $16$, $7$, $50$, $34$, $45$, $8$) 
%transitive subgroups of $S_3$ (resp. $S_4$, $S_5$, $S_6$, $S_7$, 
%$S_8$, $S_9$, $S_{10}$, $S_{11}$) 
(see Butler and McKay \cite{BM83} for $n\leq 11$, 
Royle \cite{Roy87} for $n=12$, 
Butler \cite{But93} for $n=14,15$ 
and \cite{GAP}).

We provide the following GAPalgorithm to certify 
whether $F=[J_{G/H}]^{fl}$ is invertible (resp. zero) 
(see also Hoshi and Yamasaki \cite[Chapter 5]{HY17}). 
Some related programs are available from\\ 
{\tt https://www.math.kyoto-u.ac.jp/\~{}yamasaki/Algorithm/RatProbNorm1Tori/}.\\

\begin{algorithm}[{see Hoshi and Yamasaki \cite[Chapter 5 and Chapter 8]{HY17}}]\label{AL}
~\\
\hspace*{3mm}
{\rm (0)} Construction of the Chevalley module $J_{G/H}$ (see \cite[Chapter 8]{HY17}): 

{\tt Norm1TorusJ(}$n,m${\tt )} returns 
$J_{G/H}$ for $G=nTm\leq S_n$ and $H$ is the stabilizer of one of the 
letters in $G$.

{\rm (1)} Whether $F=[J_{G/H}]^{fl}$ is invertible: 

{\tt IsInvertibleF(Norm1TorusJ(}$n,m${\tt ))} returns 
true (resp. false) if $[J_{G/H}]^{fl}$ is invertible 
(resp. not invertible) for $G=nTm\leq S_n$ and $H$ is the stabilizer of one of 
the letters in $G$ (see \cite[Section 5.2]{HY17}). 

%However, if the rank of $F$ is not small, 
%this function does not work well (in a suitable time). 
%In this paper, we provide and use the following new function 
%which works well for more higher rank cases and 
%is based on \cite[Section 5.7]{HY17}: 
%
%{\tt IsInvertibleFFromBase(Norm1TorusJ(}$n,m${\tt ),}$m_i${\tt )} 
%returns the same as {\tt IsInvertibleF(Norm1TorusJ(}$n,m${\tt ))} 
%but with respect to a given suitable $m_i=P^\circ$ 
%with the coflabby resolution 
%$0\rightarrow{\rm Ker}f\rightarrow P^\circ\xrightarrow{f}M^\circ\rightarrow 0$ %
%of $M^\circ$ 
%which may be obtained by the function {\tt SearchCoflabbyResolutionBase} 
%(see \cite[Section 5.7]{HY17} and also Example \ref{ex15} and Example \ref{ex22}).

{\rm (2)} Possibility for $F=0$ where $F=[J_{G/H}]^{fl}$: 

{\tt PossibilityOfStablyPermutationF(Norm1TorusJ(}$m,n${\tt ))} 
returns a basis $\mathcal{L}=\{l_1,\ldots,l_s\}$ of possible solutions 
space $\{(a_1,\ldots,a_r,b_1)\}$ $(a_i,b_1\in\bZ)$ (see also \cite[Section 5.4]{HY17}) to 
\begin{align*}
\bigoplus_{i=1}^r \bZ[G/H_i]^{\oplus a_i}\ \simeq\ F^{\oplus(-b_1)}
\end{align*}
for $G=mTn\leq S_n$, $H$ is the stabilizer of one of 
the letters in $G$ and $F=[J_{G/H}]^{fl}$. 
In particular, if all the $b_1$'s are even, then 
we can conclude that $F=[J_{G/H}]^{fl}\neq 0$. 

{\rm (3)} Verification of $F=0$ where $F=[J_{G/H}]^{fl}$:

{\tt FlabbyResolutionLowRankFromGroup((Norm1TorusJ(}$n,m${\tt ),TransitiveGroup(}$n,m${\tt)).actionF} 
returns a suitable flabby class $F=[J_{G/H}]^{fl}$ of $J_{G/H}$ 
with low rank for $G=nTm\leq S_n$ and $H$ is the stabilizer of one of 
the letters in $G$ 
by using the backtracking techniques. 
Repeating the algorithm, 
by defining $[J_{G/H}]^{fl^n}:=[[J_{G/H}]^{fl^{n-1}}]^{fl}$ inductively, 
$[J_{G/H}]^{fl}=0$ is provided if we may find some $n$ with 
$[J_{G/H}]^{fl^n}=0$ (this method is slightly improved to 
the {\tt flfl} algorithm, see \cite[Section 5.3]{HY17}). 
\end{algorithm}

{\it Proof of Theorem \ref{thmain2}.} 
We may assume that 
$H$ is the stabilizer of one of the letters in $G$ 
(see the first paragraph of Section \ref{seProof2}).\\

%%%%%%%%%%The case 10Tm
(1) The case $10Tm$ $(1 \leq m\leq 45)$. 

By \cite[Theorem 1.11]{HY}, 
we should show that $T$ is stably $k$-rational 
for $10T11\simeq A_5\times C_2$. 
For $10T11$, by Algorithm \ref{AL} (3), 
we may take $F=[J_{G/H}]^{fl}$ with ${\rm rank}_\bZ(F)=31$, 
$F^\prime=[F]^{fl}$ with ${\rm rank}_\bZ(F^\prime)=13$ 
and $F^{\prime\prime}=[F^\prime]^{fl}$ with $F^{\prime\prime}=[\bZ]=0$. 
This implies that $F=0$ 
and hence $T$ is stably $k$-rational (see Example \ref{ex10}).\\

%%%%%%%%%%The case 12Tm
(2) The case $12Tm$ $(1\leq m \leq 301)$. 

(2-1) The case where $K/k$ is Galois: $1 \leq m \leq 5$. 
For $12T1 \simeq C_{15}$, $12T2 \simeq C_6 \times C_2$, $12T3 \simeq D_6$, $12T4 \simeq A_4$, $12T5 \simeq C_3 \rtimes C_4$, $K/k$ is a Galois extension. 
By 
%Theorem \ref{th13-2}, 
Endo and Miyata 
\cite[Theorem 2.3]{EM75}, 
$T$ is stably $k$-rational for $12T1$, $12T5$. 
By 
%Theorem \ref{th13-1}, 
Endo and Miyata 
\cite[Theorem 1.5]{EM75},  
$T$ is not retract $k$-rational for $12T2$, $12T3$, $12T4$. 

(2-2) The case where $K/k$ is not Galois: $6 \leq m \leq 301$. 

Case 1: $m=11$. For $12T11 \simeq C_4\times S_3$, 
%For $G=12T11$, 
by Algorithm \ref{AL} (3), 
we may take $F=[J_{G/H}]^{fl}$ with ${\rm rank}_\bZ(F)=17$, 
$F^\prime=[F]^{fl}$ with ${\rm rank}_\bZ(F^\prime)=4$ 
and $F^\prime$ is permutation. 
This implies that $F=0$ and 
hence $T$ is stably $k$-rational (see Example \ref{ex12}). 
(We note that $12T1\leq 12T5\leq 12T11$.)

Case 2: $m\neq 11$. 
By using the command\\ 
{\tt List([1..301],x->Filtered([1..x],y->IsSubgroup(TransitiveGroup(12,x),}\\
{\tt TransitiveGroup(12,y))))}\\ 
in GAP \cite{GAP} (see also Example \ref{ex14} for the case where $n=14$), 
we obtain the inclusions $12Tm\leq 12Tm^\prime$ 
among the groups $G=12Tm$ with minimal groups $12Tm$ 
where $m\in I_{12}:=\{2,3,4,7,8,9,12,15,16,17,19,29,30,31,32,\\
33,34,36,40,41,46,47,57,58,59,60,61,63,64,65,66,68,69,70,73,74,75,76,89,91,93,96,99,100,102,105,107,\\
160,162,166,171,172,173,179,181,182,183,207,212,216,246,254,272,278,295\}$.

By using the command\\ 
{\tt Filtered(List(ConjugacyClassesSubgroups(TransitiveGroup(12,}$m${\tt )),Representative),\\
x->Length(Orbits(x,[1..12]))=1)},\\ 
we also see the following inclusions for 
$12Tm$ with $m\in I_0:=\{207, 212, 216,254,272,278,295\}$ 
(see Example \ref{ex12}, 
we may reduce these cases which take more computational time and resources): \\ 
$12T166\leq  12T207, 12T254$,\\ 
$12T46\leq 12T212, 12T216, 12T272$,\\
$12T17\leq 12T278$,\\
$12T2\leq 12T295$.

By the inclusion of $G=12Tm$ above 
and Lemma \ref{lemp3} (ii), 
it is enough to check that $[J_{G/H}]^{fl}$ is not invertible 
for $I_{12}\setminus I_0$. 
By Algorithm \ref{AL} (1), we obtain that $[J_{G/H}]^{fl}$ is not invertible 
and hence, by Theorem \ref{thEM73} (iii), 
$T$ is not retract $k$-rational 
for $m\in I_{12}\setminus I_0$ (see Example \ref{ex12}).\\

%%%%%%%%%%The case 14Tm
(3) The case $14Tm$ $(1\leq m\leq 63)$.

(3-1) The case where $K/k$ is Galois: $m=1,2$. 
For $14T1 \simeq C_{14}$ and $14T2 \simeq D_{7}$, $K/k$ is a Galois extension. 
By 
%Theorem \ref{th13-2}, 
Endo and Miyata 
\cite[Theorem 2.3]{EM75}, 
$T$ is stably $k$-rational for $14T1$ and $14T2$. 

(3-2) The case where $K/k$ is not Galois: $3 \leq m \leq 63$.

Case 1: $m = 3$. For $14T3 \simeq D_{14}$, 
by Algorithm \ref{AL} (1), we obtain that 
$[J_{G/H}]^{fl}$ is invertible and hence $T$ is retract $k$-rational 
by Theorem \ref{thEM73} (iii). 
By Algorithm \ref{AL} (3), 
we may take $F=[J_{G/H}]^{fl}$ with ${\rm rank}_\bZ(F)=17$ 
and $F^\prime=[F]^{fl}=\bZ^2$ which is permutation. 
This implies that $F=0$ and 
hence $T$ is stably $k$-rational by Theorem \ref{thEM73} (i) 
(see Example \ref{ex14}).

Case 2: $m = 4,5,7,16,19,46,47,49$. 
By Algorithm \ref{AL} (1), 
we see that $[J_{G/H}]^{fl}$ is invertible and hence 
$T$ is retract $k$-rational by Theorem \ref{thEM73} (iii) 
for $m = 4,5,7,16,19,46,47,49$. 
For $m=4,5,16$, 
by Algorithm \ref{AL} (2), 
we see that $[J_{G/H}]^{fl}\neq 0$ 
and hence $T$ is not stably $k$-rational (see Example \ref{ex14}). 
By Lemma \ref{lemp3} (i) and 
the inclusions $14T4 \leq 14T7, 14T46$ and 
$14T5 \leq 14T19 \leq 14T47 \leq 14T49$, 
we have 
$[J_{G/H}]^{fl}\neq 0$ 
and hence $T$ is also not stably $k$-rational for $m=7,19,46,47,49$.

Case 3: $6 \leq m \leq 63$ and $m \not = 7,16,19,46,47,49$. 

By using the command\\ 
{\tt List([1..63],x->Filtered([1..x],y->IsSubgroup(TransitiveGroup(14,x),}\\
{\tt TransitiveGroup(14,y))))}\\ 
in GAP \cite{GAP} (see Example \ref{ex14}), 
we get the inclusions $14Tm\leq 14Tm^\prime$ 
among the groups $G=14Tm$ with minimal groups $14Tm$ 
where $m\in I_{14}:=\{6,8,10,12,26,30\}$.

By the inclusion of $G=14Tm$ above and Lemma \ref{lemp3} (ii), 
it is enough to show that $[J_{G/H}]^{fl}$ is not invertible 
for $m\in I_{14}$. 
By Algorithm \ref{AL} (1), we see that 
$[J_{G/H}]^{fl}$ is not invertible 
and hence $T$ is not retract $k$-rational 
for $m\in I_{14}$ (see Example \ref{ex14}).\\

%%%%%%%%%%The case 15Tm
(4) The case $15Tm$ $(1\leq m \leq 104)$.

(4-1) The case where $K/k$ is Galois: $m=1$. 
For $15T1 \simeq C_{15}$, $K/k$ is a Galois extension. 
It follows from 
%Theorem \ref{th13-2} 
Endo and Miyata 
\cite[Theorem 2.3]{EM75} 
that $T$ is stably $k$-rational for $15T1$. 

(4-2) The case where $K/k$ is not Galois: $2 \leq m \leq 104$. 

Case 1: $m=2,3,4$. 
For $15T2 \simeq D_{15}, 
15T3 \simeq D_5 \times C_3, 15T4 \simeq S_3 \times C_5$, 
it follows from 
%Theorem \ref{th15} 
Endo 
\cite[Theorem 3.1]{End11} 
that 
$T$ is stably $k$-rational for $15T2, 15T3, 15T4$.

Case 2: $m=5,7,10,16,23$.
By Algorithm \ref{AL} (1), 
we see that $[J_{G/H}]^{fl}$ is invertible and hence 
$T$ is retract $k$-rational for $m =5, 7, 16, 23$.

For $15T5\simeq A_5$, by Algorithm \ref{AL} (3), 
we get $F=[J_{G/H}]^{fl}$ with ${\rm rank}_\bZ(F)=21$ 
and 
$F^\prime=[F]^{fl}=\bZ$. 
This implies that $F=0$ and 
hence $T$ is stably $k$-rational (see Example \ref{ex15}). 

For $15T7 \simeq D_5 \times S_3, 
15T16 \simeq A_5 \times C_3, 
15T23 \simeq A_5 \times S_3$, 
it is enough to prove that $[J_{G/H}]^{fl}=0$ 
for $G=15T23$ because $15T7\leq 15T23$, $15T16\leq 15T23$ 
and Lemma \ref{lemp3} (i). 
By Algorithm \ref{AL} (3), 
we obtain that $F=[J_{G/H}]^{fl}$ with ${\rm rank}_\bZ(F)=27$, 
$F^\prime=[F]^{fl}$ with ${\rm rank}_\bZ(F^\prime)=8$ 
and $F^{\prime\prime}=[F^\prime]^{fl}$ with $F^{\prime\prime}=\bZ$. 
This implies that $F=0$ and 
hence $T$ is stably $k$-rational (see Example \ref{ex15}). 

For $15T10\simeq S_5$, 
by Algorithm \ref{AL} (2), 
we obtain that $[J_{G/H}]^{fl}\neq 0$ and hence 
$T$ is not stably $k$-rational (see Example \ref{ex15}). 

Case 3: $m=6,8,11,22,24,29$. 
For 
$15T6\simeq C_{15}\rtimes C_4$, 
$15T8 \simeq F_{20} \times C_3$, 
it follows from 
%Theorem \ref{th15} 
Endo 
\cite[Theorem 3.1]{End11} 
that $[J_{G/H}]^{fl}$ is invertible and $[J_{G/H}]^{fl}\neq 0$. 
Hence $T$ is not stably but retract $k$-rational.

For $m=11,22,24,29$,  by Algorithm \ref{AL} (1), 
we see that $[J_{G/H}]^{fl}$ is invertible and hence 
$T$ is retract $k$-rational. 
By Lemma \ref{lemp3} (i) and the inclusions 
$15T6 \leq 15T11,15T22,15T29$ and 
$15T8 \leq 15T24$, 
we obtain that $[J_{G/H}]^{fl}\neq 0$ and hence 
$T$ is not stably $k$-rational for $m=11,22,24,29$. 

Case 4: $9\leq m\leq 104$ and $m\neq 10,11,16,22,23,24,29$. 

By using the command\\ 
{\tt List([1..104],x->Filtered([1..x],y->IsSubgroup(TransitiveGroup(15,x),}\\
{\tt TransitiveGroup(15,y))))}\\ 
in GAP \cite{GAP} (see also Example \ref{ex14} for the case where $n=14$), 
we obtain the inclusions $15Tm\leq 15Tm^\prime$ 
among the groups $G=15Tm$ with minimal groups $15Tm$ 
where $m\in I_{15}:=\{9,15,20,26\}$.

By the inclusions of groups $G=15Tm$ above 
and Lemma \ref{lemp3} (ii), 
it is enough to show that $[J_{G/H}]^{fl}$ is not invertible 
%$T$ is not retract $k$-rational 
for $m\in I_{15}$. 
By Algorithm \ref{AL} (1), we obtain that 
$[J_{G/H}]^{fl}$ is not invertible and hence 
$T$ is not retract $k$-rational for $m\in I_{15}$ 
(see Example \ref{ex15}).\qed\\

We give GAP \cite{GAP} computations in the proof of Theorem \ref{thmain2} 
for $n=10, 12,14,15$ 
in Example \ref{ex10} to Example \ref{ex15} 
(see \cite[Chapter 5]{HY17} for the explanation of the functions). 
Some related programs are available from 
{\tt https://www.math.kyoto-u.ac.jp/\~{}yamasaki/Algorithm/RatProbNorm1Tori/}.\\

\begin{example}[Computations for $10T11\leq S_{10}$]\label{ex10}
~\\
\begin{verbatim}
gap> Read("FlabbyResolutionFromBase.gap");

gap> J:=Norm1TorusJ(10,11);
<matrix group with 3 generators>
gap> StructureDescription(J);
"C2 x A5"
gap> IsInvertibleF(J); # 10T11 is retract k-rational
true
gap> T:=TransitiveGroup(10,11);
A(5)[x]2
gap> F:=FlabbyResolutionLowRankFromGroup(J,T).actionF;
<matrix group with 3 generators>
gap> Rank(F.1); # F is of rank 31
31
gap> F2:=FlabbyResolutionLowRankFromGroup(F,T).actionF;
<matrix group with 3 generators>
gap> Rank(F2.1); # [F]^fl is of rank 13
13
gap> F3:=FlabbyResolutionLowRankFromGroup(F2,T).actionF; 
# 10T11 is stably k-rational because [F]^fl=0 
Group([ [ [ 1 ] ], [ [ 1 ] ], [ [ 1 ] ] ])
\end{verbatim}
~\\
\end{example}
 
\begin{example}[Computations for $12Tm\leq S_{12}$]\label{ex12}
~\\
\begin{verbatim}
gap> Read("FlabbyResolutionFromBase.gap");

gap> J:=Norm1TorusJ(12,11);
<matrix group with 3 generators>
gap> StructureDescription(J);
"C4 x S3"
gap> IsInvertibleF(J); # 12T11 is retract k-rational
true
gap> T:=TransitiveGroup(12,11);
S(3)[x]C(4)
gap> F:=FlabbyResolutionLowRankFromGroup(J,T).actionF;
<matrix group with 3 generators>
gap> Rank(F.1); # F is of rank 17
17
gap> F2:=FlabbyResolutionLowRankFromGroup(F,T).actionF;
<matrix group with 3 generators>
gap> Rank(F2.1); # [F]^fl is of rank 4
4
gap> GeneratorsOfGroup(F2);
[ [ [ 1, 0, 0, 0 ], [ 0, 1, 0, 0 ], [ 0, 0, 1, 0 ], [ 0, 0, 0, 1 ] ], 
  [ [ 1, 0, 0, 0 ], [ 0, 1, 0, 0 ], [ 0, 0, 1, 0 ], [ 0, 0, 0, 1 ] ], 
  [ [ 0, 1, 1, 2 ], [ 0, 1, 0, 0 ], [ 3, -3, -2, -6 ], [ -1, 1, 1, 3 ] ] ]
gap> F3:=FlabbyResolutionLowRankFromGroup(F2,T).actionF; 
# 12T11 is stably k-rational because [F]^fl is permutation
[  ]

gap> IsInvertibleF(Norm1TorusJ(12,2)); # 12T2 is not retract k-rational
false
gap> IsInvertibleF(Norm1TorusJ(12,3)); # 12T3 is not retract k-rational
false
gap> IsInvertibleF(Norm1TorusJ(12,4)); # 12T4 is not retract k-rational
false
gap> IsInvertibleF(Norm1TorusJ(12,7)); # 12T7 is not retract k-rational
false
gap> IsInvertibleF(Norm1TorusJ(12,8)); # 12T8 is not retract k-rational
false
gap> IsInvertibleF(Norm1TorusJ(12,9)); # 12T9 is not retract k-rational
false
gap> IsInvertibleF(Norm1TorusJ(12,12)); # 12T12 is not retract k-rational
false
gap> IsInvertibleF(Norm1TorusJ(12,15)); # 12T15 is not retract k-rational
false
gap> IsInvertibleF(Norm1TorusJ(12,16)); # 12T16 is not retract k-rational
false
gap> IsInvertibleF(Norm1TorusJ(12,17)); # 12T17 is not retract k-rational
false
gap> IsInvertibleF(Norm1TorusJ(12,19)); # 12T19 is not retract k-rational
false
gap> IsInvertibleF(Norm1TorusJ(12,29)); # 12T29 is not retract k-rational
false
gap> IsInvertibleF(Norm1TorusJ(12,30)); # 12T30 is not retract k-rational
false
gap> IsInvertibleF(Norm1TorusJ(12,31)); # 12T31 is not retract k-rational
false
gap> IsInvertibleF(Norm1TorusJ(12,32)); # 12T32 is not retract k-rational
false
gap> IsInvertibleF(Norm1TorusJ(12,33)); # 12T33 is not retract k-rational
false
gap> IsInvertibleF(Norm1TorusJ(12,34)); # 12T34 is not retract k-rational
false
gap> IsInvertibleF(Norm1TorusJ(12,36)); # 12T36 is not retract k-rational
false
gap> IsInvertibleF(Norm1TorusJ(12,40)); # 12T40 is not retract k-rational
false
gap> IsInvertibleF(Norm1TorusJ(12,41)); # 12T41 is not retract k-rational
false
gap> IsInvertibleF(Norm1TorusJ(12,46)); # 12T46 is not retract k-rational
false
gap> IsInvertibleF(Norm1TorusJ(12,47)); # 12T47 is not retract k-rational
false
gap> IsInvertibleF(Norm1TorusJ(12,57)); # 12T57 is not retract k-rational
false
gap> IsInvertibleF(Norm1TorusJ(12,58)); # 12T58 is not retract k-rational
false
gap> IsInvertibleF(Norm1TorusJ(12,59)); # 12T59 is not retract k-rational
false
gap> IsInvertibleF(Norm1TorusJ(12,60)); # 12T60 is not retract k-rational
false
gap> IsInvertibleF(Norm1TorusJ(12,61)); # 12T61 is not retract k-rational
false
gap> IsInvertibleF(Norm1TorusJ(12,63)); # 12T63 is not retract k-rational
false
gap> IsInvertibleF(Norm1TorusJ(12,64)); # 12T64 is not retract k-rational
false
gap> IsInvertibleF(Norm1TorusJ(12,65)); # 12T65 is not retract k-rational
false
gap> IsInvertibleF(Norm1TorusJ(12,66)); # 12T66 is not retract k-rational
false
gap> IsInvertibleF(Norm1TorusJ(12,68)); # 12T68 is not retract k-rational
false
gap> IsInvertibleF(Norm1TorusJ(12,69)); # 12T69 is not retract k-rational
false
gap> IsInvertibleF(Norm1TorusJ(12,70)); # 12T70 is not retract k-rational
false
gap> IsInvertibleF(Norm1TorusJ(12,73)); # 12T73 is not retract k-rational
false
gap> IsInvertibleF(Norm1TorusJ(12,74)); # 12T74 is not retract k-rational
false
gap> IsInvertibleF(Norm1TorusJ(12,75)); # 12T75 is not retract k-rational
false
gap> IsInvertibleF(Norm1TorusJ(12,76)); # 12T76 is not retract k-rational
false
gap> IsInvertibleF(Norm1TorusJ(12,89)); # 12T89 is not retract k-rational
false
gap> IsInvertibleF(Norm1TorusJ(12,91)); # 12T91 is not retract k-rational
false
gap> IsInvertibleF(Norm1TorusJ(12,93)); # 12T93 is not retract k-rational
false
gap> IsInvertibleF(Norm1TorusJ(12,96)); # 12T96 is not retract k-rational
false
gap> IsInvertibleF(Norm1TorusJ(12,99)); # 12T99 is not retract k-rational
false
gap> IsInvertibleF(Norm1TorusJ(12,100)); # 12T100 is not retract k-rational
false
gap> IsInvertibleF(Norm1TorusJ(12,102)); # 12T102 is not retract k-rational
false
gap> IsInvertibleF(Norm1TorusJ(12,105)); # 12T105 is not retract k-rational
false
gap> IsInvertibleF(Norm1TorusJ(12,107)); # 12T107 is not retract k-rational
false
gap> IsInvertibleF(Norm1TorusJ(12,160)); # 12T160 is not retract k-rational
false
gap> IsInvertibleF(Norm1TorusJ(12,162)); # 12T162 is not retract k-rational
false
gap> IsInvertibleF(Norm1TorusJ(12,166)); # 12T166 is not retract k-rational
false
gap> IsInvertibleF(Norm1TorusJ(12,171)); # 12T171 is not retract k-rational
false
gap> IsInvertibleF(Norm1TorusJ(12,172)); # 12T172 is not retract k-rational
false
gap> IsInvertibleF(Norm1TorusJ(12,173)); # 12T173 is not retract k-rational
false
gap> IsInvertibleF(Norm1TorusJ(12,179)); # 12T179 is not retract k-rational
false
gap> IsInvertibleF(Norm1TorusJ(12,181)); # 12T181 is not retract k-rational
false
gap> IsInvertibleF(Norm1TorusJ(12,182)); # 12T182 is not retract k-rational
false
gap> IsInvertibleF(Norm1TorusJ(12,183)); # 12T183 is not retract k-rational
false
gap> IsInvertibleF(Norm1TorusJ(12,246)); # 12T246 is not retract k-rational
false

gap> List(Filtered(List(ConjugacyClassesSubgroups(TransitiveGroup(12,207)),
> Representative),x->Length(Orbits(x,[1..12]))=1),Size);
[ 576, 1152 ]
gap> List(Filtered(List(ConjugacyClassesSubgroups(TransitiveGroup(12,212)),
> Representative),x->Length(Orbits(x,[1..12]))=1),Size);
[ 72, 72, 72, 72, 72, 144, 144, 648, 648, 1296 ]
gap> List(Filtered(List(ConjugacyClassesSubgroups(TransitiveGroup(12,216)),
> Representative),x->Length(Orbits(x,[1..12]))=1),Size);
[ 72, 72, 72, 648, 648, 1296 ]
gap> List(Filtered(List(ConjugacyClassesSubgroups(TransitiveGroup(12,254)),
> Representative),x->Length(Orbits(x,[1..12]))=1),Size);
[ 576, 576, 1152, 1728, 3456 ]
gap> List(Filtered(List(ConjugacyClassesSubgroups(TransitiveGroup(12,272)),
> Representative),x->Length(Orbits(x,[1..12]))=1),Size);
[ 72, 72, 144, 720, 7920 ]
gap> List(Filtered(List(ConjugacyClassesSubgroups(TransitiveGroup(12,278)),
> Representative),x->Length(Orbits(x,[1..12]))=1),Size);
[ 12, 12, 24, 36, 36, 72, 72, 144, 576, 14400 ]
gap> List(Filtered(List(ConjugacyClassesSubgroups(TransitiveGroup(12,295)),
> Representative),x->Length(Orbits(x,[1..12]))=1),Size);
[ 12, 12, 12, 24, 24, 24, 36, 48, 48, 60, 72, 72, 72, 96, 96, 96, 120, 120, 
  144, 192, 216, 240, 432, 660, 720, 720, 1440, 7920, 95040 ]
\end{verbatim}
~\\
\end{example}
 
\begin{example}[Computations for $14Tm\leq S_{14}$]\label{ex14}
~\\
\begin{verbatim}
gap> Read("FlabbyResolutionFromBase.gap");

gap> List([1..63],x->Filtered([1..x],y->IsSubgroup(TransitiveGroup(14,x),
> TransitiveGroup(14,y))));
[ [ 1 ], [ 2 ], [ 1, 2, 3 ], [ 2, 4 ], [ 1, 5 ], [ 6 ], [ 1, 2, 3, 4, 5, 7 ],
  [ 1, 8 ], [ 1, 6, 9 ], [ 10 ], [ 6, 11 ], [ 12 ], [ 1, 2, 3, 8, 13 ],
  [ 1, 5, 8, 14 ], [ 1, 8, 15 ], [ 16 ], [ 1, 5, 10, 17 ],
  [ 1, 5, 6, 9, 11, 18 ], [ 1, 5, 19 ], [ 1, 2, 3, 8, 12, 13, 20 ], [ 6, 21 ],
  [ 12, 22 ], [ 12, 23 ], [ 1, 2, 3, 4, 5, 7, 8, 13, 14, 24 ],
  [ 1, 2, 3, 8, 13, 15, 25 ], [ 26 ], [ 2, 6, 21, 27 ], [ 6, 21, 28 ],
  [ 1, 6, 9, 21, 29 ], [ 30 ], [ 1, 2, 3, 8, 12, 13, 15, 20, 22, 25, 31 ],
  [ 1, 2, 3, 4, 5, 7, 8, 12, 13, 14, 20, 23, 24, 32 ], [ 6, 11, 33 ],
  [ 6, 10, 11, 34 ], [ 6, 11, 21, 35 ], [ 12, 22, 23, 36 ],
  [ 1, 2, 3, 4, 5, 7, 8, 13, 14, 15, 24, 25, 37 ],
  [ 1, 2, 3, 6, 9, 21, 27, 28, 29, 38 ], [ 30, 39 ],
  [ 2, 4, 6, 11, 21, 27, 35, 40 ], [ 6, 11, 21, 28, 35, 41 ],
  [ 1, 5, 6, 9, 11, 18, 33, 42 ], [ 1, 5, 6, 9, 10, 11, 17, 18, 19, 34, 43 ],
  [ 1, 5, 6, 9, 11, 18, 21, 29, 35, 44 ],
  [ 1, 2, 3, 4, 5, 7, 8, 12, 13, 14, 15, 20, 22, 23, 24, 25, 31, 32, 36, 37, 45
     ], [ 2, 4, 46 ], [ 1, 5, 19, 47 ],
  [ 1, 2, 3, 4, 5, 6, 7, 9, 11, 18, 21, 27, 28, 29, 35, 38, 40, 41, 44, 48 ],
  [ 1, 2, 3, 4, 5, 7, 19, 46, 47, 49 ], [ 6, 10, 11, 21, 33, 34, 35, 50 ],
  [ 1, 5, 6, 9, 10, 11, 17, 18, 19, 21, 29, 33, 34, 35, 42, 43, 44, 50, 51 ],
  [ 1, 5, 8, 14, 15, 16, 19, 52 ], [ 6, 10, 11, 21, 33, 34, 35, 50, 53 ],
  [ 2, 4, 6, 10, 11, 21, 27, 33, 34, 35, 40, 46, 50, 53, 54 ],
  [ 6, 10, 11, 21, 28, 33, 34, 35, 41, 50, 53, 55 ],
  [ 1, 5, 6, 9, 10, 11, 17, 18, 19, 21, 29, 33, 34, 35, 42, 43, 44, 47, 50, 51,
      53, 56 ],
  [ 1, 2, 3, 4, 5, 6, 7, 9, 10, 11, 17, 18, 19, 21, 27, 28, 29, 33, 34, 35, 38,
      40, 41, 42, 43, 44, 46, 47, 48, 49, 50, 51, 53, 54, 55, 56, 57 ],
  [ 1, 5, 8, 14, 15, 16, 19, 26, 47, 52, 58 ], [ 12, 22, 23, 36, 59 ],
  [ 1, 2, 3, 4, 5, 7, 8, 13, 14, 15, 16, 19, 24, 25, 26, 37, 46, 47, 49, 52,
      58, 60 ],
  [ 1, 2, 3, 4, 5, 7, 8, 12, 13, 14, 15, 16, 19, 20, 22, 23, 24, 25, 26, 31,
      32, 36, 37, 45, 46, 47, 49, 52, 58, 59, 60, 61 ],
  [ 6, 10, 11, 12, 21, 22, 23, 28, 30, 33, 34, 35, 36, 41, 50, 53, 55, 59, 62 ],
  [ 1, 2, 3, 4, 5, 6, 7, 8, 9, 10, 11, 12, 13, 14, 15, 16, 17, 18, 19, 20, 21,
      22, 23, 24, 25, 26, 27, 28, 29, 30, 31, 32, 33, 34, 35, 36, 37, 38, 39,
      40, 41, 42, 43, 44, 45, 46, 47, 48, 49, 50, 51, 52, 53, 54, 55, 56, 57,
      58, 59, 60, 61, 62, 63 ] ]

gap> IsInvertibleF(Norm1TorusJ(14,4)); # 14T4 is retract k-rational
true
gap> IsInvertibleF(Norm1TorusJ(14,5)); # 14T5 is retract k-rational
true
gap> IsInvertibleF(Norm1TorusJ(14,7)); # 14T7 is retract k-rational
true
gap> IsInvertibleF(Norm1TorusJ(14,16)); # 14T16 is retract k-rational
true
gap> IsInvertibleF(Norm1TorusJ(14,19)); # 14T19 is retract k-rational
true
gap> IsInvertibleF(Norm1TorusJ(14,46)); # 14T46 is retract k-rational
true
gap> IsInvertibleF(Norm1TorusJ(14,47)); # 14T47 is retract k-rational
true
gap> IsInvertibleF(Norm1TorusJ(14,49)); # 14T49 is retract k-rational
true

gap> PossibilityOfStablyPermutationF(Norm1TorusJ(14,4)); 
# 14T4 is not stably k-rational by Algorithm 4.1 (2)
[ [ 1, 1, 1, 1, 1, -1, -1, 1, -2 ] ]
gap> PossibilityOfStablyPermutationF(Norm1TorusJ(14,5)); 
# 14T5 is not stably k-rational by Algorithm 4.1 (2)
[ [ 2, -1, 0, 3, 0, 1, 0, -1, -2 ] ]
gap> PossibilityOfStablyPermutationF(Norm1TorusJ(14,16)); 
# 14T16 in not stably k-rational by Algorithm 4.1 (2)
[ [ 1, 0, 0, 0, 0, 0, -2, -1, -3, -4, -2, 0, 0, -3, 0, 2, 2, 0, 1, 2, 2, 1, -4, 4 ], 
  [ 0, 1, 0, 0, 0, 0, -1, 0, -1, -1, 0, 0, 0, -1, 0, 1, 0, 1, 0, 1, 1, 0, -1, 0 ], 
  [ 0, 0, 1, 0, 0, 0, -1, -1, -1, -2, -1, 0, 0, -1, 0, 1, 1, -1, 1, 1, 0, 0, -1, 2 ], 
  [ 0, 0, 0, 1, 0, 0, -1, 0, -1, 0, 0, 0, 0, 0, 0, 1, 0, 1, -1, -1, 1, 1, -1, 0 ], 
  [ 0, 0, 0, 0, 1, 0, 1, -1, 0, -4, -2, 0, 0, -3, -1, -1, 2, -3, 2, 3, -1, 0, -1, 4 ] ]

gap> J:=Norm1TorusJ(14,3);
<matrix group with 2 generators>
gap> StructureDescription(J);
"D28"
gap> IsInvertibleF(J); # 14T3 is retract k-rational
true
gap> T:=TransitiveGroup(14,3);
D(7)[x]2
gap> F:=FlabbyResolutionLowRankFromGroup(J,T).actionF; 
<matrix group with 2 generators>
gap> Rank(F.1); # F is of rank 17
17
gap> F2:=FlabbyResolutionLowRankFromGroup(F,T).actionF; 
# 14T3 is stably k-rational because [F]^fl=0
Group([ [ [ 1, 0 ], [ 0, 1 ] ], [ [ 1, 0 ], [ 0, 1 ] ] ])

gap> IsInvertibleF(Norm1TorusJ(14,6)); # 14T6 is not retract k-rational
false
gap> IsInvertibleF(Norm1TorusJ(14,8)); # 14T8 is not retract k-rational
false
gap> IsInvertibleF(Norm1TorusJ(14,10)); # 14T10 is not retract k-rational
false
gap> IsInvertibleF(Norm1TorusJ(14,12)); # 14T12 is not retract k-rational
false
gap> IsInvertibleF(Norm1TorusJ(14,26)); # 14T26 is not retract k-rational
false
gap> IsInvertibleF(Norm1TorusJ(14,30)); # 14T30 is not retract k-rational
false
\end{verbatim}
~\\
\end{example}

\begin{example}[Computations for $15Tm\leq S_{15}$]\label{ex15} 
~\\
\begin{verbatim}
gap> Read("FlabbyResolutionFromBase.gap");

gap> J:=Norm1TorusJ(15,5);
<matrix group with 2 generators>
gap> StructureDescription(J);
"A5"
gap> IsInvertibleF(J); # 15T5 is retract k-rational
true
gap> T:=TransitiveGroup(15,5);
A_5(15)
gap> F:=FlabbyResolutionLowRankFromGroup(J,T).actionF;
<matrix group with 2 generators>
gap> Rank(F.1); # F is of rank 21
21
gap> F2:=FlabbyResolutionLowRankFromGroup(F,T).actionF;
# 15T5 is stably k-rational because [F]^fl=0
Group([ [ [ 1 ] ], [ [ 1 ] ] ])

gap> J:=Norm1TorusJ(15,23);
<matrix group with 3 generators>
gap> StructureDescription(J);
"A5 x S3"
gap> IsInvertibleF(J); # 15T23 is retract k-rational
true
gap> T:=TransitiveGroup(15,23);
A(5)[x]S(3)
gap> F:=FlabbyResolutionLowRankFromGroup(J,T).actionF;
<matrix group with 3 generators>
gap> Rank(F.1); # F is of rank 27
27
gap> F2:=FlabbyResolutionLowRankFromGroup(F,T).actionF;
<matrix group with 3 generators>
gap> Rank(F2.1); # [F]^fl is of rank 8
8
gap> F3:=FlabbyResolutionLowRankFromGroup(F2,T).actionF;
# 15T23 is stably k-rational because [[F]^fl]^fl=0
Group([ [ [ 1 ] ], [ [ 1 ] ], [ [ 1 ] ] ])

gap> IsInvertibleF(Norm1TorusJ(15,10)); # 15T10 is retract k-rational
true
gap> PossibilityOfStablyPermutationF(Norm1TorusJ(15,10)); 
# 15T10 is not stably k-rational by Algorithm 4.1 (2) 
[ [ 1, 0, 0, 0, 0, 0, 8, 1, -2, 5, -3, 2, 2, 5, 0, -8, -10, -3, 8, -2 ], 
  [ 0, 1, 0, 0, 0, -1, -1, 0, 0, -1, 0, 0, 0, 0, 1, 1, 1, 0, -1, 0 ], 
  [ 0, 0, 1, 0, 0, 0, -2, 0, 0, -1, 1, 0, -1, -1, 0, 2, 2, 1, -2, 0 ], 
  [ 0, 0, 0, 1, 0, 0, 12, 1, -3, 7, -5, 2, 3, 6, 0, -12, -14, -4, 12, -2 ], 
  [ 0, 0, 0, 0, 1, 2, -2, 0, 1, -2, 2, -2, -1, -2, -2, 2, 4, 1, -2, 0 ] ]

gap> IsInvertibleF(Norm1TorusJ(15,11)); # 15T11 is retract k-rational
true
gap> IsInvertibleF(Norm1TorusJ(15,22)); # 15T22 is retract k-rational
true
gap> IsInvertibleF(Norm1TorusJ(15,24)); # 15T24 is retract k-rational
true
gap> IsInvertibleF(Norm1TorusJ(15,29)); # 15T29 is retract k-rational
true

gap> IsInvertibleF(Norm1TorusJ(15,9)); # 15T9 is not retract k-rational
false
gap> IsInvertibleF(Norm1TorusJ(15,15)); # 15T15 is not retract k-rational
false
gap> IsInvertibleF(Norm1TorusJ(15,20)); # 15T20 is not retract k-rational
false
gap> IsInvertibleF(Norm1TorusJ(15,26)); # 15T26 is not retract k-rational
false
\end{verbatim}
\end{example}
%%%%%%%%%%%%%%%%%%%%%%%%%%%%%%%%%%%%%%%%%%%%
\section{Proof of Theorem {\ref{thmain3}}, Theorem {\ref{thmain4}} and {Theorem \ref{thmain5}}}\label{seProof34}

{\it Proof of Theorem \ref{thmain3}.} 
We may assume that 
$H$ is the stabilizer of one of the letters in $G$ 
(see the first paragraph of Section \ref{seProof2}).

Step 1. 
It is enough to show that $F=[J_{G/H}]^{fl}$ is not invertible 
for $G={\rm PSL}_2(\bF_q)$ because 
$\PSL_2(\bF_q)\leq G\leq {\rm P\Gamma L}_2(\bF_q)$ 
and Lemma \ref{lemp3} (ii). 
The group 
$G={\rm PSL}_2(\bF_q)$ acts on $\bP^1(\bF_q)=\bF_q\cup\{\infty\}$ 
via linear fractional transformation. 
Let $\bF_q^\times=\langle u\rangle$. 
Then $\bP^1(\bF_q)=\bF_q^\times\cup\{0\}\cup\{\infty\}$ 
and 
$\bF_q^\times=\{1,-1,\sqrt{-1},-\sqrt{-1},u^i,-u^i,u^{-i},-u^{-i}\mid 1\leq i\leq \frac{q-5}{4}\}$ because $q\equiv 1\pmod{4}$.

Step 2. 
Take a subgroup $V_4=\langle\sigma,\tau\rangle\simeq C_2\times C_2
\leq G={\rm PSL}_2(\bF_q)$ as 
\begin{align*}
\sigma=\begin{pmatrix} \sqrt{-1} & 0 \\ 0 & -\sqrt{-1} \end{pmatrix},\ 
\tau=\begin{pmatrix} 0 & -1 \\ 1 & 0 \end{pmatrix}.
\end{align*}
The action of $V_4=\langle\sigma,\tau\rangle$ 
on $\bP^1(\bF_q)$ is given as 
$\sigma:x\mapsto -x$ and $\tau:x\mapsto -1/x$. 
This action induces the action of $V_4$ on $J_{G/H}$ given by 
\begin{align*}
&\sigma: e_1 \leftrightarrow e_{-1},\ 
e_{\sqrt{-1}} \leftrightarrow e_{-\sqrt{-1}},\ 
e_{u^i} \leftrightarrow e_{-u^i},\ 
e_{-u^{-i}} \leftrightarrow e_{u^{-i}},\ 
e_0 \mapsto e_0,\ 
e_\infty\mapsto e_\infty,\\
&\tau:e_1 \leftrightarrow e_{-1},\ 
e_{\pm \sqrt{-1}} \mapsto e_{\pm \sqrt{-1}},\ 
e_{u^i} \leftrightarrow e_{-u^{-i}},\ 
e_{-u^i} \leftrightarrow e_{u^{-i}},\ 
e_0 \leftrightarrow e_\infty,\\
&\sigma\tau:e_{\pm 1} \mapsto e_{\pm 1},\ 
e_{\sqrt{-1}} \leftrightarrow e_{-\sqrt{-1}},\ 
e_{u^i} \leftrightarrow e_{u^{-i}},\ 
e_{-u^i} \leftrightarrow e_{-u^{-i}},\ 
e_0 \leftrightarrow e_\infty
\end{align*}
where 
$B=\{e_1,e_{-1},e_{\sqrt{-1}},e_{-\sqrt{-1}},e_{u^i},e_{-u^i},e_{u^{-i}},e_{-u^{-i}},e_0\mid 1\leq i\leq \frac{q-5}{4}\}$ is a $\bZ$-basis of $J_{G/H}$ and 
\[
e_\infty:=-\sum_{j\in\bF_{q}} e_j.
\]

By Lemma \ref{lemp3} (ii), 
we should show that $[M]^{fl}$ is not invertible 
where $M=J_{G/H}|_{V_4}$ is a $V_4$-lattice with ${\rm rank}_\bZ(M)=q=n-1$.

Step 3. 
We will construct a coflabby resolution 
$0\rightarrow F^\circ\rightarrow P^\circ\rightarrow M^\circ\rightarrow 0$ 
where $P^\circ$ is permutation $V_4$-lattice 
and $F^\circ$ is coflabby $V_4$-lattice with ${\rm rank}_\bZ(F^\circ)=5$.

Step 3-1. 
The actions of $\sigma$ and $\tau$ on $M$ are 
represented as matrices 
\begin{align*}
&\left(
\begin{array}{cccc|cccc|ccc|c}
0 & 1 &&&&&&&&&\\
1 & 0 &&&&&&&&&\\
&& 0 & 1 &&&&&&&\\
&& 1 & 0 &&&&&&&\\\hline
&&&& 0 & 1 & 0 & 0&&&\\
&&&& 1 & 0 & 0 & 0&&&\\
&&&& 0 & 0 & 0 & 1&&&\\
&&&& 0 & 0 & 1 & 0&&&\\\hline
&&&&&&&&\ddots&&\\
&&&&&&&&&\ddots&&\\
&&&&&&&&&&\ddots&\\\hline
&&&&&&&&&&&1
\end{array}
\right),\\
&\left(
\begin{array}{cccc|cccc|ccc|c}
0 & 1 &&&&&&&&&&\\
1 & 0 &&&&&&&&&&\\
&& 1 & 0 &&&&&&&\\
&& 0 & 1 &&&&&&&\\\hline
&&&& 0 & 0 & 0 & 1&&&\\
&&&& 0 & 0 & 1 & 0&&&\\
&&&& 0 & 1 & 0 & 0&&&\\
&&&& 1 & 0 & 0 & 0&&&\\\hline
&&&&&&&&\ddots&&\\
&&&&&&&&&\ddots&&\\
&&&&&&&&&&\ddots&\\\hline
\!-1\!&\!-1\!&\!-1\!&\!-1\!&\!\-1\!&\!-1\!&\!-1\!&\!-1\!&\!-1\!&\cdots&\!-1\!&\!-1\!
\end{array}
\right).
\end{align*}
Let $B^\ast=\{e^\ast_1,e^\ast_{-1},e^\ast_{\sqrt{-1}},e^\ast_{-\sqrt{-1}},e^\ast_{u^i},e^\ast_{-u^i},e^\ast_{u^{-i}},e^\ast_{-u^{-i}},e^\ast_0\mid 1\leq i\leq \frac{q-5}{4}\}$ 
be the dual basis of $B$. 
By the definition, $B^\ast$ is a $\bZ$-basis of the $G$-lattice $I_{G/H}=(J_{G/H})^\circ$. 
The action of $V_4=\langle\sigma,\tau\rangle$ on $M^\circ$ is given by 
\begin{align*}
&\sigma: e^\ast_1 \leftrightarrow e^\ast_{-1},\ 
e^\ast_{\sqrt{-1}} \leftrightarrow e^\ast_{-\sqrt{-1}},\ 
e^\ast_{u^i} \leftrightarrow e^\ast_{-u^i},\ 
e^\ast_{-u^{-i}} \leftrightarrow e^\ast_{u^{-i}},\ 
e^\ast_0 \mapsto e^\ast_0,\\
&\tau:e^\ast_1 \leftrightarrow e^\ast_{-1}-e^\ast_0,\ 
e^\ast_{\pm \sqrt{-1}} \leftrightarrow e^\ast_{\pm \sqrt{-1}}-e^\ast_0,\ 
e^\ast_{u^i} \leftrightarrow e^\ast_{-u^{-i}}-e^\ast_0,\ 
e^\ast_{-u^i} \leftrightarrow e^\ast_{u^{-i}}-e^\ast_0,\ 
e_0^\ast\mapsto -e^\ast_0,\\
&\sigma\tau:e^\ast_{\pm 1} \leftrightarrow e^\ast_{\pm 1}-e^\ast_0,\ 
e^\ast_{\pm\sqrt{-1}} \leftrightarrow e^\ast_{\mp\sqrt{-1}}-e^\ast_0,\ 
e^\ast_{u^i} \leftrightarrow e^\ast_{u^{-i}}-e^\ast_0,\ 
e^\ast_{-u^i} \leftrightarrow e^\ast_{-u^{-i}}-e^\ast_0,\ 
e_0^\ast\mapsto -e^\ast_0
\end{align*}
(this action corresponds to the transposed matrices of the above matrices).

We define the permutation $V_4$-lattice $P^\circ$ 
of ${\rm rank}_\bZ(P^\circ)=q+5=n+4$
with $\bZ$-basis 
\begin{align*}
&v_1:=v(e^\ast_1),\ 
v_2:=v(e^\ast_{-1}),\ 
v_3:=v(e^\ast_1-e^\ast_0), 
v_4:=v(e^\ast_{-1}-e^\ast_0),\ 
v_5:=v(e^\ast_1+e^\ast_{-1}-e^\ast_0),\\
&v_6:=v(e^\ast_{\sqrt{-1}}), 
v_7:=v(e^\ast_{-\sqrt{-1}}),
v_8:=v(e^\ast_{\sqrt{-1}}-e^\ast_0),
v_9:=v(e^\ast_{-\sqrt{-1}}-e^\ast_0),
v_{10}:=v(e^\ast_{\sqrt{-1}}+e^\ast_{-\sqrt{-1}}-e^\ast_0),\\
&v_{i,1}:=v(e^\ast_{u^i}), 
v_{i,2}:=v(e^\ast_{-u^i}),
v_{i,3}:=v(e^\ast_{u^{-i}}-e^\ast_0),
v_{i,4}:=v(e^\ast_{-u^{-i}}-e^\ast_0)\quad 
(1\leq i\leq \tfrac{q-5}{4})
\end{align*}
where $V_4$ acts on $P^\circ$ 
by $g(v(m^\ast))=v(g(m^\ast))$ $(m^\ast\in M^\circ, g\in V_4)$: 
\begin{align*}
\sigma:&\ v_1\leftrightarrow v_2,\ 
v_3\leftrightarrow v_4,\ 
v_5\mapsto v_5,\ 
v_6\leftrightarrow v_7,\ 
v_8\leftrightarrow v_9,\ 
v_{10}\mapsto v_{10},\ 
v_{i,1}\leftrightarrow v_{i,2},\ 
v_{i,3}\leftrightarrow v_{i,4},\\
\tau:&\ v_1\leftrightarrow v_4,\ 
v_2\leftrightarrow v_3,\ 
v_5\mapsto v_5,\ 
v_6\leftrightarrow v_9,\ 
v_7\leftrightarrow v_8,\ 
v_{10}\mapsto v_{10},\ 
v_{i,1}\leftrightarrow v_{i,4},\ 
v_{i,2}\leftrightarrow v_{i,3},\\
\sigma\tau:&\ v_1\leftrightarrow v_3,\ 
v_2\leftrightarrow v_4,\ 
v_5\mapsto v_5,\ 
v_6\leftrightarrow v_8,\ 
v_7\leftrightarrow v_9,\ 
v_{10}\mapsto v_{10},\ 
v_{i,1}\leftrightarrow v_{i,3},\ 
v_{i,2}\leftrightarrow v_{i,4}.
\end{align*}

Step 3-2. 
We define a $V_4$-homomorphism 
$f:P^\circ\rightarrow M^\circ$, $v(m^\ast)\mapsto m^\ast$ 
$(m^\ast\in M^\circ)$. 
Then $f$ is surjective. 
We define a $V_4$-lattice $F^\circ$ as $F^\circ={\rm Ker}(f)$. 
Then we obtain an exact sequence 
$0\rightarrow F^\circ\rightarrow P^\circ\rightarrow M^\circ\rightarrow 0$ 
with ${\rm rank}_\bZ(F^\circ)=5$.

Step 3-3. 
We will check that $F^\circ$ is coflabby. 
In order to prove this assertion, 
we should check that 
$\widetilde{f}=f|_{H^0(W,P^\circ)}:H^0(W,P^\circ)\rightarrow H^0(W,M^\circ)$ 
is surjective (hence $H^1(W,F^\circ)=0$) for any $W\leq V_4$ 
where $H^0(W,P^\circ)=\widehat{Z}^0(W,P^\circ)=(P^\circ)^W$ 
(see also \cite[Chapter 2]{HY17}). 

Step 3-3-1. $W=V_4=\langle\sigma,\tau\rangle$. 
By the orbit decomposition of the action of $V_4$ on $P^\circ$, 
\begin{align*}
\{v_1+v_2+v_3+v_4,v_5,v_6+v_7+v_8+v_9,v_{10},v_{i,1}+v_{i,2}+v_{i,3}+v_{i,4}\mid 1\leq i\leq \tfrac{q-5}{4}\}
\end{align*}
is a $\bZ$-basis of $(P^\circ)^{V_4}$. 
We also see that 
\begin{align*}
\{e^\ast_1+e^\ast_{-1}-e^\ast_0,
e^\ast_{\sqrt{-1}}+e^\ast_{-\sqrt{-1}}-e^\ast_0,
e^\ast_{u^i}+e^\ast_{-u^i}+e^\ast_{u^{-i}}+e^\ast_{-u^{-i}}
-2e^\ast_0\mid 1\leq i\leq \tfrac{q-5}{4}\}
\end{align*}
is a $\bZ$-basis of 
$(M^\circ)^{V_4}$. 
Hence $\widetilde{f}$ is surjective because 
\begin{align*}
\widetilde{f}:\ &v_1+v_2+v_3+v_4\mapsto 2(e^\ast_1+e^\ast_{-1}-e^\ast_0),\ 
v_5\mapsto e^\ast_1+e^\ast_{-1}-e^\ast_0,\\ 
&v_6+v_7+v_8+v_9\mapsto 2(e^\ast_{\sqrt{-1}}+e^\ast_{-\sqrt{-1}}-e^\ast_0),\ 
v_{10}\mapsto e^\ast_{\sqrt{-1}}+e^\ast_{-\sqrt{-1}}-e^\ast_0,\\
&v_{i,1}+v_{i,2}+v_{i,3}+v_{i,4}\mapsto e^\ast_{u^i}+e^\ast_{-u^i}+e^\ast_{u^{-i}}+e^\ast_{-u^{-i}}
-2e^\ast_0\quad (1\leq i\leq \tfrac{q-5}{4}).
\end{align*}

Step 3-3-2. $W=\langle\sigma\rangle$. 
The set 
\begin{align*}
\{v_1+v_2,v_3+v_4,v_5,v_6+v_7,v_8+v_9,v_{10},v_{i,1}+v_{i,2},v_{i,3}+v_{i,4}\mid 1\leq i\leq \tfrac{q-5}{4}\}
\end{align*}
becomes a $\bZ$-basis 
of $(P^\circ)^{\langle\sigma\rangle}$ and 
\begin{align*}
\{e^\ast_1+e^\ast_{-1},
e^\ast_{\sqrt{-1}}+e^\ast_{-\sqrt{-1}},
e^\ast_{u^i}+e^\ast_{-u^i},e^\ast_{u^{-i}}+e^\ast_{-u^{-i}},e^\ast_0\mid 1\leq i\leq \tfrac{q-5}{4}\}
\end{align*}
is a $\bZ$-basis of $(M^\circ)^{\langle\sigma\rangle}$. 
Hence $\widetilde{f}$ is surjective because 
\begin{align*}
\widetilde{f}:\ &v_1+v_2\mapsto e^\ast_1+e^\ast_{-1}, 
v_5\mapsto e^\ast_1+e^\ast_{-1}-e^\ast_0,\ 
v_6+v_7\mapsto e^\ast_{\sqrt{-1}}+e^\ast_{-\sqrt{-1}},\\
&v_{i,1}+v_{i,2}\mapsto e^\ast_{u^i}+e^\ast_{-u^i},\ 
v_{i,3}+v_{i,4}\mapsto e^\ast_{u^{-i}}+e^\ast_{-u^{-i}}
-2e^\ast_0\quad (1\leq i\leq \tfrac{q-5}{4}).
\end{align*}

Step 3-3-3. $W=\langle\tau\rangle$. 
The set 
\begin{align*}
\{v_1+v_4,v_2+v_3,v_5,v_6+v_8,v_7+v_9,v_{10},v_{i,1}+v_{i,4},v_{i,2}+v_{i,3}\mid 1\leq i\leq \tfrac{q-5}{4}\}
\end{align*}
becomes a $\bZ$-basis of $(P^\circ)^{\langle\tau\rangle}$ and 
\begin{align*}
\{e^\ast_1+e^\ast_{-1}-e^\ast_0,
e^\ast_{\sqrt{-1}}+e^\ast_{-\sqrt{-1}}-e^\ast_0, 
2e^\ast_{-\sqrt{-1}}-e^\ast_0, 
e^\ast_{u^i}+e^\ast_{-u^{-i}}-e^\ast_0, 
e^\ast_{u^{-i}}+e^\ast_{-u^{i}}-e^\ast_0\mid 1\leq i\leq \tfrac{q-5}{4}\}
\end{align*}
is a $\bZ$-basis of $(M^\circ)^{\langle\tau\rangle}$. 
Hence $\widetilde{f}$ is surjective because 
\begin{align*}
\widetilde{f}:\ &v_5\mapsto e^\ast_1+e^\ast_{-1}-e^\ast_0,\ 
v_7+v_9\mapsto 2e^\ast_{-\sqrt{-1}}-e^\ast_0,\ 
v_{10}\mapsto e^\ast_{\sqrt{-1}}+e^\ast_{-\sqrt{-1}}-e^\ast_0,\\
&v_{i,1}+v_{i,4}\mapsto e^\ast_{u^i}+e^\ast_{-u^{-i}}-e^\ast_0,\ 
v_{i,2}+v_{i,3}\mapsto e^\ast_{-u^i}+e^\ast_{u^{-i}}-e^\ast_0\quad (1\leq i\leq \tfrac{q-5}{4}).
\end{align*}

Step 3-3-4. $W=\langle\sigma\tau\rangle$. 
The set 
\begin{align*}
\{v_1+v_3,v_2+v_4,v_5,v_6+v_9,v_7+v_8,v_{10},v_{i,1}+v_{i,3},v_{i,2}+v_{i,4}\mid 1\leq i\leq \tfrac{q-5}{4}\}
\end{align*}
becomes a $\bZ$-basis of $(P^\circ)^{\langle\sigma\tau\rangle}$ and 
\begin{align*}
\{e^\ast_1+e^\ast_{-1}-e^\ast_0, 2e^\ast_{-1}-e^\ast_0, 
e^\ast_{\sqrt{-1}}+e^\ast_{-\sqrt{-1}}-e^\ast_0, 
e^\ast_{u^i}+e^\ast_{u^{-i}}-e^\ast_0, 
e^\ast_{-u^i}+e^\ast_{-u^{-i}}-e^\ast_0\mid 1\leq i\leq \tfrac{q-5}{4}\}
\end{align*}
is a $\bZ$-basis of 
$(M^\circ)^{\langle\sigma\tau\rangle}$. 
Hence $\widetilde{f}$ is surjective because 
\begin{align*}
\widetilde{f}:\ &v_5\mapsto e^\ast_1+e^\ast_{-1}-e^\ast_0,\ 
v_2+v_4\mapsto 2e^\ast_{-1}-e^\ast_0,\ 
v_{10}\mapsto e^\ast_{\sqrt{-1}}+e^\ast_{-\sqrt{-1}}-e^\ast_0,\\
&v_{i,1}+v_{i,3}\mapsto e^\ast_{u^i}+e^\ast_{u^{-i}}-e^\ast_0,\ 
v_{i,2}+v_{i,4}\mapsto e^\ast_{-u^i}+e^\ast_{-u^{-i}}-e^\ast_0\quad (1\leq i\leq \tfrac{q-5}{4}).
\end{align*}

Step 4. We will prove that $F$ is not invertible. 
By Step 3, we have an exact sequence 
$0\rightarrow F^\circ\rightarrow P^\circ\rightarrow M^\circ\rightarrow 0$ 
where $P^\circ$ is permutation $V_4$-lattice 
and $F^\circ$ is coflabby $V_4$-lattice with ${\rm rank}_\bZ(F^\circ)=5$. 

The set $\{w_1,w_2,w_3,w_4,w_5\}$ becomes a $\bZ$-basis of $F^\circ$ where 
\begin{align*}
&w_1=v_1+v_4-v_5,\ 
w_2=v_2-v_4+v_8+v_9-v_{10},\ 
w_3=v_3+v_4-v_5-v_8-v_9+v_{10},\\
&w_4=v_6+v_9-v_{10},\ 
w_5=v_7+v_8-v_{10}.
\end{align*} 
The actions of $\sigma$ and $\tau$ on $F^\circ$ are 
given by 
\begin{align*}
&\sigma: 
w_1\mapsto w_2+w_3,\ 
w_2\mapsto w_1-w_3,\ 
w_3\mapsto w_3,\ 
w_4\leftrightarrow w_5,\\
&\tau: 
w_1\mapsto w_1,\ 
w_2\mapsto -w_1+w_3+w_4+w_5,\ 
w_3\mapsto w_1+w_2-w_4-w_5,\ 
w_4\leftrightarrow w_5
\end{align*}
and they are represented as matrices 
\begin{align*}
\left(
\begin{array}{ccccc}
 0 & 1 & 1 & 0 & 0 \\
 1 & 0 &-1 & 0 & 0 \\
 0 & 0 & 1 & 0 & 0 \\
 0 & 0 & 0 & 0 & 1 \\
 0 & 0 & 0 & 1 & 0
\end{array}
\right),
\left(
\begin{array}{ccccc}
 1 & 0 & 0 & 0 & 0 \\
-1 & 0 & 1 & 1 & 1 \\
 1 & 1 & 0 &-1 &-1 \\
 0 & 0 & 0 & 0 & 1 \\
 0 & 0 & 0 & 1 & 0
\end{array}
\right).
\end{align*}

By taking the dual, we get the flabby resolution 
$0 \rightarrow M \rightarrow P \rightarrow F \rightarrow 0$
of $M$ and the actions of $\sigma$ and $\tau$ on $F$ are 
represented as the following matrices (transposed matrices of the above): 
\begin{align*}
S=
\left(
\begin{array}{ccccc}
0 & 1 & 0 & 0 & 0 \\
1 & 0 & 0 & 0 & 0 \\
1 &-1 & 1 & 0 & 0 \\
0 & 0 & 0 & 0 & 1 \\
0 & 0 & 0 & 1 & 0
\end{array}
\right),\ 
T=
\left(
\begin{array}{ccccc}
1 &-1 & 1 & 0 & 0 \\
0 & 0 & 1 & 0 & 0 \\
0 & 1 & 0 & 0 & 0 \\
0 & 1 &-1 & 0 & 1 \\
0 & 1 &-1 & 1 & 0
\end{array}
\right).
\end{align*}

In order to obtain $H^1(V_4,F)$, 
we should evaluate the elementary divisors of 
\begin{align*}
(S-I\mid T-I)=
\left(
\begin{array}{ccccc|ccccc}
-1 & 1 & 0 & 0 & 0 & 0 &-1 & 1 & 0 & 0 \\
 1 &-1 & 0 & 0 & 0 & 0 &-1 & 1 & 0 & 0 \\
 1 &-1 & 0 & 0 & 0 & 0 & 1 &-1 & 0 & 0 \\
 0 & 0 & 0 &-1 & 1 & 0 & 1 &-1 &-1 & 1 \\
 0 & 0 & 0 & 1 &-1 & 0 & 1 &-1 & 1 &-1
\end{array}
\right)
\end{align*}
where $I$ is the $5\times 5$ identity matrix. 
Multiply the regular matrix 
\begin{align*}
Q=
\left(
\begin{array}{ccccc}
-1 & 0 & 0 & 0 & 0 \\
 0 & 0 & 0 & 0 & 1 \\
 0 &-1 & 1 & 0 & 0 \\
 1 & 0 & 1 & 0 & 0 \\
 1 & 1 & 0 & 1 & 1
\end{array}
\right)
\end{align*}
from the left, we have 
\begin{align*}
Q\,(S-I\mid T-I)=
\left(
\begin{array}{ccccc|ccccc}
 1 &-1 & 0 & 0 & 0 & 0 & 1 &-1 & 0 & 0 \\
 0 & 0 & 0 & 1 &-1 & 0 & 1 &-1 & 1 &-1 \\
 0 & 0 & 0 & 0 & 0 & 0 & 2 &-2 & 0 & 0 \\
 0 & 0 & 0 & 0 & 0 & 0 & 0 & 0 & 0 & 0 \\
 0 & 0 & 0 & 0 & 0 & 0 & 0 & 0 & 0 & 0
\end{array}
\right).
\end{align*}
Hence we conclude that $H^1(V_4,F)=\bZ/2\bZ$. 
This implies that $F$ is not invertible. \qed\\

%%%%%%%%%%%%%%%%%%%%%%%%%%%%%%%%%%%%%%%%%%%%%%

Let $p$ be a prime number and 
$G \leq S_{2p}$ be a primitive %transitive 
subgroup. 
Wielandt (\cite{Wie56}, \cite{Wie64}) 
proved that $G$ is doubly transitive if $2p-1$ is not a perfect square. 
Using the classification of finite simple groups, 
all doubly transitive finite groups are known 
(see Cameron \cite[Theorem 5.3]{Cam81} and also Dixon and Mortimer 
\cite[Section 7.7]{DM96}). 
On the other hand, by 
O'Nan-Scott theorem (see Liebeck, Praeger and Saxl \cite{LPS88}), 
$G$ must be almost simple, i.e. $S\leq G\leq {\rm Aut}(S)$ 
for some non-abelian simple group $S$. 
The socle ${\rm soc}(G)\lhd G$ of a group $G$ was classified by 
Liebeck and Saxl \cite[Theorem 1.1 (i), (iii)]{LS85}.
\begin{theorem}[{Liebeck and Saxl \cite[Corollary 1.2]{LS85}, see also \cite[Theorem 4.6]{Sha97}, \cite[Proposition 5.5]{DJ13}}]
Let $p$ be a prime number and $G \leq S_{2p}$ be a primitive %transitive 
subgroup. 
Then $G$ is one of the following:\\
{\rm (i)} $G=S_{2p}$ or $G=A_{2p}\leq S_{2p}$;\\
{\rm (ii)} $G=S_5\leq S_{10}$ or $G=A_5\leq S_{10}$;\\
{\rm (iii)} $G=M_{22}\leq S_{22}$ 
or $G=\Aut(M_{22})\simeq M_{22}\rtimes C_2\leq S_{22}$ 
where $M_{22}$ is the Mathieu group of degree $22$;\\ 
{\rm (iv)} $\PSL_2(\bF_q)\leq G\leq {\rm P\Gamma L}_2(\bF_q)\simeq 
\PGL_2(\bF_q)\rtimes C_e$
%{\rm Gal}(\bF_q/\bF_l)$ 
%where $2p=q+1$ and $q=l^{2^a}$ for some odd prime number $l$ and $a\geq 1$.
where $2p=q+1$ and $q=l^e$ is an odd prime power.\\
\end{theorem}

{\it Proof of Theorem \ref{thmain4}.} 
We may assume that 
$H$ is the stabilizer of one of the letters in $G$ 
(see the first paragraph of Section \ref{seProof2}). 

(i) follows from 
%Theorem \ref{thS} (i) 
Cortella and Kunyavskii 
\cite[Proposition 0.2]{CK00}
and 
%Theorem \ref{thA} (i). 
Endo 
\cite[Theorem 5.2]{End11}. 

(ii) follows from Theorem \ref{thmain2} (1) 
because $S_5\simeq 10T13$ and $A_5\simeq 10T7$. 
%\cite[Theorem 1.11 (iv)]{HY}. 

For (iii), it is enough to show that $F=[J_{G/H}]^{fl}$ is not invertible 
for $G=M_{22}\leq S_{22}$. 
%By Algorithm \ref{AL} (1), 
%we obtain that $F=[J_{G/H}]^{fl}$ with ${\rm rank}_\bZ(F)=672$ 
We see that there exists $G^\prime\leq G$ such that $[J_{G/H}|_{G^\prime}]^{fl}$ is not invertible. 
Indeed, we can find such $G^\prime$ which is isomorphic to $(C_2)^3$, 
$Q_8$, $D_4$ or $C_4\times C_2$ (see Example \ref{ex22}). 
Hence it follows from Lemma \ref{lemp3} (ii) that $F$ is not invertible. 
This implies that $T$ is not retract $k$-rational 
by Theorem \ref{thEM73} (iii). %(see Example \ref{ex22}). 

For (iv), we may assume that $p\geq 3$ 
(if $p=2$, then $q=3$ and 
$\PSL_2(\bF_3)\simeq A_4$, $\PGL_2(\bF_3)\simeq S_4$, see (i)). 
%see also Theorem \ref{thS} and Theorem \ref{thA}). 
Then $q=2p-1\equiv 1\pmod{4}$ because $p$ is odd. 
Hence the assertion follows from Theorem \ref{thmain3} 
as a special case where $n=2p$ and $q=l^e$.\qed\\

%%%%%%%%%%%%%%%%%%%%%%%%%%%%%%%%%%%%%%%%%%%%%%%%%%%%%%%%%%%%%%%%%%%%
{\it Proof of Theorem \ref{thmain5}.} 
The assertion for $n=11$ and $n=23$ follows from 
\cite[Theorem 1.9 (6)]{HY}. 
The assertion for $n=12$ and $n=22$ follows from 
Theorem \ref{thmain2} (2)--(ii) 
and Theorem \ref{thmain4} (iii) respectively. 
%if $G={\rm Gal}(L/k)\leq S_n$ is the Mathieu group $M_n$ 
%and $H={\rm Gal}(K/K)$ is the stabilizer of one of the letters in $G$, 
%then the norm one torus 
%$R_{K/k}^{(1)}(\bG_m)$ is not retract $k$-rational. 
%We finally treat the last case $M_{24}$ of the Mathieu group:

Let $G=M_{24}$ be the Mathieu group of degree $24$. 
Then there exists $G^\prime\leq G\leq S_{24}$ which is transitive 
and isomorphic to $S_4$ (see Example \ref{ex22}). 
Then $[J_{G^\prime}]^{fl}$ is not invertible by 
%Theorem \ref{th13-1}.
Endo and Miyata 
\cite[Theorem 1.5]{EM75}. 
It follows from Lemma \ref{lemp3} (ii) that $[J_{G/H}]^{fl}$ 
is not invertible and hence 
$R_{K/k}^{(1)}(\bG_m)$ is not retract $k$-rational 
by Theorem \ref{thEM73} (iii).\qed\\

%%%%%%%%%%%%%%%%%%%%%%%%%%%%%%%%%%%%%%%%%%%%%%%%%%%%%%%%%%
\begin{example}[Computations for $22T38\simeq M_{22}\leq S_{22}$ and $M_{24}\leq S_{24}$]\label{ex22}
~\\
\begin{verbatim}
gap> Read("FlabbyResolutionFromBase.gap");

gap> JM22:=Norm1TorusJ(22,38);
<matrix group with 2 generators>
gap> StructureDescription(JM22);
"M22"
gap> M22:=TransitiveGroup(22,38);
gap> M22s:=List(ConjugacyClassesSubgroups2(M22),Representative);;
gap> JM22s:=ConjugacyClassesSubgroups2FromGroup(JM22,M22);;
gap> JM22s8:=Filtered(JM22s,x->Size(x)=8);;
gap> Length(JM22s8);
12
gap> JM22s8false:=Filtered(JM22s8,x->IsInvertibleF(x)=false);;
# [J_{G/H}|G']^fl is not invertible
gap> List(JM22s8false,StructureDescription);
[ "C2 x C2 x C2", "Q8", "D8", "C4 x C2" ]

gap> M24:=PrimitiveGroup(24,1);
M(24)
gap> M24s:=Filtered(List(ConjugacyClassesSubgroups2(M24),Representative),
> x->Length(Orbits(x,[1..24]))=1 and Size(x)=24);;
gap> M24s4:=Filtered(M24s,x->IdGroup(x)=[24,12]);;
gap> List(M24s4,StructureDescription);
[ "S4", "S4", "S4" ]
\end{verbatim}
\end{example}

%%%%%%%%%%%%%%%%%%%%%%%%%%%%%%%%%%%%%%%%%%%%%%%%%%%%%%%%%%%%%%%%%%%%%%%%%%

\end{document}